\documentclass{article}
\usepackage[table]{xcolor}
\usepackage[utf8]{inputenc}
\usepackage{tikz-cd}
\usepackage{amsmath}
\usepackage{verbatim}
\usepackage{amssymb}
\usepackage{amsthm}
\usepackage{hyperref}
\usepackage{a4wide}
\usepackage{multirow}
\usepackage{bigdelim}
\usepackage[bb=dsserif]{mathalpha}

\newtheorem*{theorem*}{Theorem}

\newcommand{\puhh}{{quasi-here\-ditary} }
\usepackage{ulem}

\usetikzlibrary{calc}
\usetikzlibrary{decorations.pathmorphing}
\tikzset{curve/.style={settings={#1},to path={(\tikztostart)
    .. controls ($(\tikztostart)!\pv{pos}!(\tikztotarget)!\pv{height}!270:(\tikztotarget)$)
    and ($(\tikztostart)!1-\pv{pos}!(\tikztotarget)!\pv{height}!270:(\tikztotarget)$)
    .. (\tikztotarget)\tikztonodes}},
    settings/.code={\tikzset{quiver/.cd,#1}
        \def\pv##1{\pgfkeysvalueof{/tikz/quiver/##1}}},
    quiver/.cd,pos/.initial=0.35,height/.initial=0}

\tikzset{tail reversed/.code={\pgfsetarrowsstart{tikzcd to}}}
\tikzset{2tail/.code={\pgfsetarrowsstart{Implies[reversed]}}}
\tikzset{2tail reversed/.code={\pgfsetarrowsstart{Implies}}}

\tikzset{no body/.style={/tikz/dash pattern=on 0 off 1mm}}
\DeclareMathOperator{\topp}{top}
\DeclareMathOperator{\Top}{top}
\DeclareMathOperator{\im}{im}

\DeclareMathOperator{\field}{k}
\DeclareMathOperator{\op}{op}
\DeclareMathOperator{\rad}{rad}

\DeclareMathOperator{\dual}{D}

\DeclareMathOperator{\modu}{mod}
\DeclareMathOperator{\Ext}{Ext}
\DeclareMathOperator{\Aut}{Aut}
\DeclareMathOperator{\Mat}{Mat}
\DeclareMathOperator{\End}{End}
\DeclareMathOperator{\Hom}{Hom}
\DeclareMathOperator{\add}{add}

\DeclareMathOperator{\Sim}{Sim}

\DeclareMathOperator{\id}{id}

\DeclareMathOperator{\soc}{soc}
\DeclareMathOperator{\tr}{tr}

\DeclareMathOperator{\Gl}{Gl}

\DeclareMathOperator{\fakedelta}{\widehat{\Delta}}
\DeclareMathOperator{\fakepi}{\widehat{\pi}}

\newtheorem{theorem}{Theorem}[section]
\newtheorem{definition}[theorem]{Definition}
\newtheorem{example}[theorem]{Example}
\newtheorem{corollary}[theorem]{Corollary}
\newtheorem{lemma}[theorem]{Lemma}
\newtheorem{remark}[theorem]{Remark}
\newtheorem{proposition}[theorem]{Proposition}
\begin{document}
\title{Quasi-hereditary Skew Group Algebras}
\author{Anna Rodriguez}
\maketitle
\section*{Introduction} 
Quasi-hereditary algebras are algebras equipped with a partial order on the isomorphism classes of simples which fulfills certain additional properties. 
They were first introduced by Scott in \cite{Scott}, and then became a central notion in the theory of highest weight categories, initiated by Cline, Parshall and Scott in \cite{CPS}.
The primary motivation of \cite{CPS} came from the theory of representations of semisimple algebraic groups. Among natural examples of \puhh algebras arising from this area are the Schur algebras of symmetric groups and algebras underlying blocks of Bernstein-Gelfand-Gelfand category $\mathcal{O}$ associated to a semisimple complex Lie algebra $\mathfrak{g}$.\\
Many other families of \puhh algebras of significant interest come from representation theory of finite-dimensional algebras itself. Among these are all finite-dimensional algebras of global dimension less than or equal to two, in particular path algebras of quivers and Auslander algebras.\\
Again in analogy to Bernstein-Gelfand-Gelfand category $\mathcal{O}$, given a \puhh algebra $A$, there is a set of $A$-modules known as the {\it standard modules} over $A$, which mimics the structure and properties of Verma modules over a semisimple complex Lie algebra $\mathfrak{g}$. Similarly to the setting of category $\mathcal{O}$, the category $\mathcal{F}(\Delta)$ of $A$-modules admitting a filtration by standard modules is of particular interest. By the Dlab-Ringel reconstruction theorem \cite{DlabRingel}, the algebra $A$ together with its quasi-hereditary structure can be reconstructed from $\mathcal{F}(\Delta)$.\\
Recall that Verma modules over $\mathfrak{g}$ are defined by induction from simple finite-dimensional modules over a Borel subalgebra.
In analogy, Koenig introduced in \cite{Koenig} the concept of an exact Borel subalgebra $B$ of a \puhh algebra $A$, which is a directed subalgebra of $A$ such that, in particular, inducing simple $B$-modules yields standard $A$-modules, so that we obtain a bijection between isomorphism classes of simple $B$-modules and standard $A$-modules. Additionally, one requires the induction functor to be exact, whence the name, which enables us to transfer homological information from $\modu B$ to $\mathcal{F}(\Delta)$. Together, these conditions allow us to describe the structure of $\mathcal{F}(\Delta)$ using an exact Borel subalgebra.\\
A fundamental theorem in the study of exact Borel subalgebras of \puhh algebras is that of Koenig, Külshammer and Ovsienko, proved in \cite{KKO}, which states that every \puhh algebra is Morita equivalent to a \puhh algebra with an exact Borel subalgebra.\\
Quasi-hereditary algebras also feature in the work of Chuang and Kessar in \cite{ChuangKessar}, which was later used by Chuang and Rouquier \cite{ChuangRouquier} in their proof of Broué's Abelian Defect Group conjecture for symmetric groups. There, the quasi-hereditary algebras that appear are Schur algebras corresponding to blocks of the group algebra of the symmetric group.
Of central interest in this setting are the RoCK blocks, which are blocks of a given weight $w$ that are Morita equivalent to the wreath product of the principal block with the symmetric group $S_w$, and their corresponding Schur algebras, which are then also Morita equivalent to the wreath product of the Schur algebra corresponding to the principal block with the symmetric group, see \cite[Theorem 5.1]{ChuangTan}.
This enables Chuang and Rouquier to use previous results by Chuang and Tan in \cite{ChuangTan2} on the wreath products of \puhh algebras with symmetric groups, something which was later again studied by Chan in \cite{Chan}.\\
In a more recent article by Evseev and Kleshchev \cite{EvseevKleshchev}, which generalizes the result of Chuang and Rouquier from the group algebra of the symmetric group to arbitrary Hecke algebras, wreath products of quasi-hereditary algebras appear again in a similar role, this time replacing the Schur algebras with zigzag algebras.\\
Recall that the wreath product algebra $A \wr S_{n}$ of an algebra $A$ with a symmetric group $S_{n}$ is isomorphic to the skew group algebra $A^{\otimes n} \ast S_{n}$. Thus, one can hope that, after additional investigation of the structure of tensor products of quasi-hereditary algebras, results on skew group algebras of quasi-hereditary algebras may be applied to wreath product algebras of quasi-hereditary algebras.\\
A skew group algebra $A \ast G$ is an algebra constructed from an algebra $A$ with an action by a group $G$ in the following way:
\begin{itemize}
  \item As a $\field$-vector space, $A*G:=A\otimes_{\field}\field G$.
  \item Multiplication is given by 
      \begin{align*}
        (a\otimes g)\cdot (a'\otimes g'):=ag(a')\otimes gg'.
      \end{align*}
\end{itemize}
The structure of skew group algebras, including their Morita equivalence class, their Hochschild cohomology and their Yoneda algebra, has been studied extensively, see for example \cite{ReitenRiedtmann}, \cite{Demonet}, \cite{LeMeur1}, \cite{witherspoon}, \cite{MartinezVilla}. The preservation of various structural properties of $A$ under the skew group construction, such as global dimension, the property of being an Auslander algebra, or the property of being Calabi-Yau, has also been investigated by many authors, including \cite{skewcalabi}, \cite{LeMeur}, \cite{ReitenRiedtmann}.  \\
In this article, we examine the relation between possible \puhh structures on $A$ and those on $A \ast G$. Further, we study the relation between the exact Borel subalgebras of the two.\\
Assuming a natural compatibility of the group action with the partial order, we show that $\leq_A$ induces a partial order $\leq_{A*G}$ on the isomorphism classes of simple $A*G$-modules, and we obtain the following theorem:
\begin{theorem*}[Theorem~{\ref{thm_ghiff}}]
 The algebra $(A, \leq_A)$ is \puhh if and only if $(A*G, \leq_{A*G})$ is \puhh.
   \end{theorem*}
Moreover, again assuming compatibility with the $G$-action, we can also relate exact Borel subalgebras of $A$ and $A*G$.
\begin{theorem*}[Theorem~{\ref{thm_Borel}}]
  Let $B\subseteq A$ be a subalgebra of $A$ such that $g(B)=B$ for every $g\in G$. Then $(B, \leq_B)$ is an exact Borel subalgebra of $(A, \leq_A)$ if and only if $(B*G, \leq_{B*G})$ is an exact Borel subalgebra of $(A*G, \leq_{A*G})$.
\end{theorem*}

The structure of the article is as follows. Section~\ref{s1} contains a brief account of skew group algebras, including a description of the simple $A*G$-modules in terms of simple $A$-modules and irreducible representations of certain subgroups of $G$.
In Section~\ref{s2}, we recall some of the central results about \puhh algebras and exact Borel subalgebras.
Section~\ref{s3} is dedicated to the synthesis of the two concepts and contains our main results.
Finally, in Section~\ref{s4} we describe some exact Borel subalgebras of Auslander algebras of certain Nakayama algebras, exemplifying our methods from preceding sections.
\subsection*{Notation}
Let $\field$ be an algebraically closed field. All algebras are assumed to be finite-dimensional $\field$-algebras, and all modules are assumed to be finite-dimensional as $\field$-vector spaces. Tensor products, if not otherwise indicated, are tensor products over $\field$. We denote by $D:=\Hom_{\field}(-, \field)$ the usual $\field$-duality. \\
For a module $M$ and an indecomposable module $N$ over some algebra $A$ we write $N|M$ if and only if $N$ is isomorphic to a direct summand of $M$.\\
We denote by $\Sim(A)$ a set of representatives of the isomorphism classes of the simple $A$-modules, and for $S\in \Sim(A)$ we write $[M:S]$ for the multiplicity of $S$ in $M$.\\
Moreover, for any module $M$ we pick a projective cover $P_M$.\\
\section{Skew Group Algebras}\label{s1}
Throughout, let $A$ be a finite-dimensional algebra over $\field$ and $G$ be a finite group acting on $A$ such that $|G|$ does not divide the characteristic of $\field$. In this chapter, we will repeat some basic definitions and results about skew group algebras. For a more detailed introduction see for example \cite{Poon} and \cite{ReitenRiedtmann}.
\begin{definition}\label{def_gM}
    For an $A$-module $M$ we define $gM:=M$ as a $\field$-vector spaces together with the multiplication
    \begin{align*}
      a\cdot_{gM}m:=g^{-1}(a)m.
    \end{align*}
    Moreover, for an $A$-linear map $f:M\rightarrow N$ we define $g(f)(m):=f(m)$.\\
    In this way, every $g\in G$ gives rise to an autoequivalence 
    \begin{align*}
      \modu A\rightarrow \modu A, M\mapsto gM, f\mapsto g(f)
    \end{align*}
    such that the map $G\rightarrow \Aut(\modu A)$ is a group homomorphism.
\end{definition}
 The module $gM$ is also sometimes denoted ${}^gM$, see for example \cite[p. 235]{ReitenRiedtmann}.\\
  However, we have chosen this notation, so that we may identify $gM$ with the set of formal products $\{gm:m\in M\}$ and then be able to write 
    \begin{align*}
      a \cdot gm=g g^{-1}(a)m.
    \end{align*}

\begin{definition}\label{def_skewgroupalgebra} \cite[p. 224]{ReitenRiedtmann}
  The skew group algebra $A*G$ is defined as
  \begin{align*}
    A*G:=A\otimes \field G
  \end{align*}
  as a $\field$-vector space together with the multiplication
  \begin{align*}
    (a\otimes g)\cdot (a'\otimes g'):=ag(a')\otimes gg'.
  \end{align*}
\end{definition}
\begin{definition}\label{def_modgact}
  Let $M$ be an $A$-module. We say that $M$ has a $G$-action if there are isomorphisms of $A$-modules
  \begin{align*}
    \tr_g^M:gM\rightarrow M
  \end{align*}
  such that 
  \begin{align*}
    \tr_g^M\circ g(\tr_h^M)=\tr_{gh}^M
  \end{align*}
  for all $g, h\in M$.
\end{definition}
\begin{proposition} \cite[Proposition 4.8]{Poon}
  There is a one-to-one correspondence between modules $(M, (\tr_g^M)_{g\in G})$ with a $G$-action and $A*G$-modules given by
  \begin{align*}
    g\cdot m:=\tr_g^M(m).
  \end{align*}
  which induces an equivalence of categories between the $A*G$-modules and the $A$-modules with a $G$-action together with the $A$-linear maps compatible with this action.
\end{proposition}
\begin{definition}\label{def_kGM} 
  Let $M$ be an $A$-module. Then we define an $A*G$-module $\field G\otimes M$ via
  \begin{align*}
    g'\cdot (g\otimes m):=g'g\otimes m\\
    a(g\otimes m):=g\otimes g^{-1}(a)m
  \end{align*}
  for $g, g'\in G, m\in M, a\in A$.\\
  Moreover, if $H$ is a subgroup of $G$ and $M$ is an $A*H$ module, then $M$ is in particular a $\field H$-module, so that we can define an $A*G$-module  $\field G\otimes_{\field H} M$ in the same way.
\end{definition}
\begin{remark}
  Note that if $H$ is a normal subgroup of $G$, then the $G$-action on $A$ induces a $G$-action on $A*H$ via 
  \begin{align*}
    g(a\otimes h)= g(a)\otimes ghg^{-1}.
  \end{align*}
\end{remark}
\begin{definition}\label{def_MtimesV} 
  Let $M$ be an $A*G$-module, $V$ be a $\field G$-module. Then we define an $A*G$-module $M\otimes V$ via
  \begin{align*}
    g\cdot (m\otimes v):=gm\otimes gv\\
    a(m\otimes v):=am\otimes v
  \end{align*}
  for $g\in G, m\in M, v\in V$ and $a\in A$.\\
  If $f:M\rightarrow N$ is a homomorphism of $A*G$-modules, then 
  \begin{align*}f\otimes \id_V:M\otimes V\rightarrow N\otimes V\end{align*}
   is a homomorphism of $A*G$-modules, and
  \begin{align*}
    -\otimes V:\modu A*G\rightarrow \modu A*G, M\mapsto M\otimes V, f\mapsto f\otimes \id_V
  \end{align*}
  defines an additive functor.\\
  Moreover, note that $- \otimes (V\oplus V')\cong -\otimes V \oplus -\otimes V'$ and $-\otimes (V\otimes V')\cong (-\otimes V)\otimes V'$.
\end{definition}
\begin{definition}\label{def_I+R}
  We denote by
  \begin{align*}
    I_G:\modu A\rightarrow \modu A*G, M\mapsto \field G\otimes M, f\mapsto \id_{\field G}\otimes f
  \end{align*}
  the induction functor along $G$.\\
  We denote by
  \begin{align*}
    R_G:\modu A*G\rightarrow \modu A, M\mapsto  _{A|}M
  \end{align*}
  the canonical restriction functor.
  Moreover, if $H$ is a subgroup of $G$, we denote by 
  \begin{align*}
    I_{G/H}:\modu A*H\rightarrow \modu A*G, M\mapsto \field G\otimes_{\field H} M, f\mapsto \id_{\field G}\otimes f
  \end{align*}
  the induction functor from $A*H$ to $A*G$ and by
  \begin{align*}
    R_{G/H}:\modu A*G\rightarrow \modu A*H, M\mapsto  _{A*H|}M
  \end{align*}
  the canonical restriction functor.
\end{definition}
\begin{lemma}\label{lemma_radical} \cite[Theorem 1.1 C]{ReitenRiedtmann}
  We have $\rad(A*G)=\rad(A)\otimes \field G$
 \end{lemma}
 The content of the following proposition is essentially a compilation of results in \cite{ReitenRiedtmann}.
 However, for the sake of convenience we will give a quick proof.
\begin{proposition}\label{prop_IR+RI}
  Let $V$ be an indecomposable $\field G$-module. Then the following statements hold:
  \begin{enumerate}
    \item For any $A*G$-module $M$
    \begin{align*}
      I_GR_G(M)\cong M\otimes \field G,
    \end{align*}
    in other words,
    \begin{align*}
      \field G\otimes M\cong M\otimes \field G.
    \end{align*}
    More precisely, there is a natural equivalence
    \begin{align*}
      I_GR_G\cong -\otimes \field G
    \end{align*}
    \item For any $A$-module $M$
      \begin{align*}
      R_GI_G(M)\cong \bigoplus_{g\in G} gM.
      \end{align*}
      More precisely, there is a natural equivalence
    \begin{align*}
      R_GI_G\cong \bigoplus_{g\in G}g
    \end{align*}
    \item $R_G, I_G$ and $-\otimes V$ are additive.
    \item $R_G, I_G$ and $-\otimes V$ are exact.
    \item $R_G, I_G$ and $-\otimes V$ preserve and reflect projective modules.
    \item $R_G, I_G$ and $-\otimes V$ preserve and reflect injective modules.
    \item $R_G, I_G$ and $-\otimes V$ preserve and reflect semisimple modules.
  \end{enumerate}
\end{proposition}
\begin{proof}
  \begin{enumerate}
    \item Let $M$ be an $A*G$-module. Then we define an isomorphism
        \begin{align*}
          \alpha_M:I_GR_G(M)=\field G\otimes_{A|}M\rightarrow M\otimes \field G, g\otimes m\mapsto gm\otimes g.
        \end{align*}
        It is easy to check that this is an isomorphism of $A*G$-modules and that $\alpha=(\alpha_M)_M$ defines a natural isomorphism.
    \item Let $M$ be an $A$-module. Then we define an isomorphism
        \begin{align*}
          \beta_M:R_GI_G(M)=_{A|}\field G\otimes M\rightarrow \bigoplus_{g\in G}gM, g\otimes m\mapsto (\delta_{gg'}m)_{g'\in G}
        \end{align*}
        It is easy to check that this is an isomorphism of $A$-modules and that $\beta=(\beta_M)_M$ defines a natural isomorphism.
    \item This is obvious.
    \item $R_G$ is a restriction functor and thus exact. Moreover, since tensor products over $\field$ are exact, $I_G$ and $-\otimes V$ are exact.
    \item  Since $A$ is an $A*G$-module via $g\cdot a=g(a)$,  $$I_G(A)\cong A\otimes \field G=A*G.$$ Hence $I_G$ preserves projectives.\\ 
            On the other hand, if $M$ is an $A$-module such that $I_G(M)$ is projective, then so is $R_GI_G(M)\cong \bigoplus_{g\in G}gM$ and since $M=eM$ is a direct summand of $\bigoplus_{g\in G}gM$ this implies that $M$ is projective. 
            Similarly, 
            $$R_G(A*G)\cong \bigoplus_{g\in G}gA\cong |G|A$$ hence $R_G$ preserves projectives. Moreover, if $M$ is an $A*G$-module such that $R_GM$ is projective, then so is $I_GR_G(M)\cong M\otimes \field G$ and since $\field |\field G$, $$M\cong M\otimes \field |M\otimes \field G.$$ This implies that $M$ is projective.\\
            Finally, since $$-\otimes V|-\otimes \field G\cong I_GR_G,$$ $-\otimes V$ preserves and reflects projectives.
    \item Since $\dual \field G\cong \field G$ as a $\field G$-module, this is analogous to the previous statement replacing $A$ by $\dual A$ and $A*G$ by $\dual A*G$
    \item Let $S$ be a semisimple $A$-module. Then $\rad(A)S=(0)$ and hence
          \begin{align*}
            \rad(A*G)I_G(S)=(\rad(A)\otimes \field G)\field G\otimes S
            =\rad(A)(\field G\otimes S)\\
            =\sum_{g\in G}\field G\otimes (g(\rad(A))S)
            =\sum_{g\in G}\field G\otimes (\rad(A)S)
            =(0)
          \end{align*}
      by Lemma \ref{lemma_radical}, so that $I_G(S)$ is semisimple.\\
       On the other hand, if $S$ is a simple $A*G$ module, then
      $\rad(A*G)S=(0)$ and hence
          \begin{align*}
            \rad(A)R_G(S)=\rad(A)S
            \subseteq \rad(A*G)S
            =(0)
          \end{align*}
      so that $R_G(S)$ is semisimple.\\
      Moreover, if $S$ is an $A$-module such that $I_G(S)$ is semisimple, then so is $R_GI_G(S)$ and hence $S$, since $S$ is a direct summand of $R_GI_G(S)\cong \bigoplus_{g\in G}gS$; and if $S$ is an $A*G$-module such that $R_G(S)$ is semisimple, then so is $I_GR_G(S)$ and hence $S$, since $S$ is a direct summand of $I_GR_G(S)\cong S\otimes \field G$.
  \end{enumerate}
\end{proof}
\begin{corollary}\label{corollary_G/H}
  Let $H$ be a subgroup of $G$ and let $Z\subseteq G$ be a set of representatives of $G/H$. Then the following statements hold:
  \begin{enumerate}
    \item $I_{G/H}\circ I_H$ is naturally equivalent to $I_G$.
    \item $R_H\circ R_{G/H}$ is naturally equivalent to $R_G$.
    \item $R_{G/H}$ and $I_{G/H}$ are additive.
    \item $R_{G/H}$ and $I_{G/H}$ are exact.
    \item $R_{G/H}$ and $I_{G/H}$ preserve and reflect projective modules.
    \item $R_{G/H}$ and $I_{G/H}$ preserve and reflect injective modules.
    \item $R_{G/H}$ and $I_{G/H}$ preserve and reflect semisimple modules.
  \end{enumerate}
  Moreover, if $H$ is a normal subgroup, then we additionally have
  \begin{itemize}
    \item[8.] For any $A*H$-module $M$
    \begin{align*}
    R_{G/H}I_{G/H}(M)\cong \bigoplus_{z\in Z} zM.
    \end{align*}
    More precisely, there is a natural equivalence
  \begin{align*}
    R_GI_G\cong \bigoplus_{g\in G}g
  \end{align*}
  \end{itemize}
\end{corollary}
\begin{proof}
  \begin{itemize}
    \item[1.-3.] These are obvious.
    \item[4.] $R_{G/H}$ is a restriction functor and thus exact. Moreover, $\field H$ is semisimple, so that tensoring over $\field H$ is exact.
    \item[5.-7.] Since
    \begin{align*}
      I_{G/H}| I_{G/H}\circ (-\otimes \field H)\cong I_{G/H}\circ I_H\circ R_H\cong I_G\circ R_H
    \end{align*} and 
    \begin{align*}
      R_{G/H}| (-\otimes \field H)\circ R_{G/H}\cong I_H\circ R_H\circ R_{G/H}\cong I_{H}\circ R_G,
    \end{align*}
    this follows from Proposition \ref{prop_IR+RI} for $H$ and $G$.
    \item[8.]  This follows analogously to 2. in Proposition \ref{prop_IR+RI}
  \end{itemize}
\end{proof}
\begin{corollary}\label{corollary_dividessimples}
  Let $L'$  be a simple $A*G$-module. Then there is a simple $A$-module $L$ such that $L'|I_GL$.
  On the other hand, if $L$ is a simple $A$-module, then there is a simple $A*G$-module $L'$ such that $L|R_GL'$.
\end{corollary}
\begin{proof}
  Let $L'$ be a simple $A*G$-module. Then $R_GL'\cong \bigoplus_{L\in \Sim(A)}[R_GL':L]L$ is semisimple and $L'$ is a summand of  $$L'\otimes \field G\cong I_GR_G L'\cong \bigoplus_{L\in \Sim(A)}[R_GL':L]I_GL.$$ Since $L'$ is simple, this implies that there is some $L\in \Sim(A)$ such that $L'|I_GL$.\\
  The second statement is analogous.
\end{proof}
\subsection{An explicit description of the simples}
For $L\in \Sim(A)$ denote by $H_L$ the stabilizer of the isomorphism class of $L$ in $G$ and let $Z_L$ be a set of representatives of $G/H_L$.\\
In this subsection, we will give an explicit description of the simples of $A*G$ in terms of simple $A$-modules $L$ and simple representations of the corresponding stabilizers $H_L$, rectifying a result in \cite{MartinezVilla}. \\
This description is not needed for our main results, but will make it possible to obtain an explicit description of the standard modules of $A*G$, see Lemma \ref{lemma_deltaforbasicA*G}.\\
\begin{lemma}
  For every isomorphism class of simple $A$-modules, there exists a representative $L$ which is $H_L$-equivariant.
\end{lemma}
\begin{proof}
  Clearly, we can assume $G=H_L$. Moreover, since $\rad(A)$ acts as zero on $L$, we can assume that $A$ is semisimple. In this case $A$ is a direct product of matrix rings. Again, the matrix rings not corresponding to $L$ act via zero, so we can assume $A=\Mat_n(\field)$ and $L=\field^n$.\\ 
  Now $G$ acts on $A$ via automorphisms, but since $A$ is a matrix ring, all of these are inner, so we obtain a group homomorphism $\varphi:G\rightarrow \Gl_n(\field)$. Hence $L$ obtains the structure of an $A*H_L$-module via  $\tr_g:gL\rightarrow L, x\mapsto \varphi(g)x$.
\end{proof}
The following proposition is a rectification of Lemma 2 in \cite{MartinezVilla}.
\begin{proposition}\label{prop_simplesofA*G}
  The simple modules of $A*G$ are exactly the modules of the form $$\field G\otimes_{\field H_L}(L\otimes V)$$ for some irreducible $\field H_L$-module $V$ and an $H_L$-equivariant simple $A$-module $L$. Two modules $ \field G\otimes_{\field H_L}(L\otimes V)$ and $\field G\otimes_{\field H_{L'}}( L'\otimes W)$ of this form are isomorphic if and only if there is $g\in G$ such that $gL\cong L'$ and $gV\cong W$.
\end{proposition}
\begin{proof}
   First we show that a module of the form $ \field G\otimes_{\field H_L}(L\otimes V)$ is indecomposable if $V$ is indecomposable. First that, using Corollary \ref{corollary_G/H}, we have that
  \begin{align*}
    &\End_{A*H_L}(\field G\otimes_{\field H_L}(L\otimes V))\\ 
    \cong& \bigoplus_{z\in Z_L}\Hom_{A*H_L}(z(L\otimes V), \field G\otimes_{\field H_L}(L\otimes V))\\
    \cong& \bigoplus_{z\in Z_L}\Hom_{A*H_L}(z(L\otimes V), z(\field G\otimes_{\field H_L}(L\otimes V)))\\
    \cong&  \bigoplus_{z\in Z_L}z\Hom_{A*H_L}(L\otimes V, \field G\otimes_{\field H_L}(L\otimes V))\\
    \cong& \field G\otimes_{\field H_L}\Hom_{A*H_L}(L\otimes V, \field G\otimes_{\field H_L}L\otimes V)
  \end{align*}
  as $\field G$-modules, where the $G$-action on $\End_{A*H_L}(\field G\otimes_{\field H_L}L\otimes V)$  is given via conjugation, and the $G$-action on $\field G\otimes_{\field H_L}\Hom_{A*H_L}(L\otimes V, \field G\otimes_{\field H_L}L\otimes V)$ is given via multiplication in $\field G$.\\
  Now if $-^G$ denotes the fix point functor under the $G$-action
  \begin{align*}
    \End_{A*G}(\field G\otimes_{\field H_L}(L\otimes V))
    &=\End_{A*H_L}(\field G\otimes_{\field H_L}(L\otimes V))^G\\
    \cong \left( \field G\otimes_{\field H_L}\Hom_{A*H_L}(L\otimes V, \field G\otimes_{\field H_L}(L\otimes V))\right)^G
    &\cong  \Hom_{A*H_L}(L\otimes V, \field G\otimes_{\field H_L}(L\otimes V)).
  \end{align*}
  Note that for $z\notin H_L$, 
  \begin{align*} 
    \Hom_A(L\otimes V, z(L\otimes V))\cong \dim(V)^2\Hom_A(L, zL)=0
  \end{align*} so that for $z\notin H_L$ 
  \begin{align*}
    \Hom_{A*H_L}(L\otimes V, z(L\otimes V))\subseteq \Hom_A(L\otimes V, z(L\otimes V))=(0).
  \end{align*}
  Thus, 
  \begin{align*}
    \Hom_{A*H_L}(L\otimes V, \field G\otimes_{\field H_L}(L\otimes V))
    \cong \bigoplus_{z\in Z_L}\Hom_{A*H_L}(L\otimes V, z(L\otimes V))\\
    \cong \End_{A*H_L}(L\otimes V)
    \cong  \End_A(L\otimes V)^{H_L}
    \cong (\End_A(L)\otimes \End_{\field}(V))^{H_L}.
  \end{align*}
  Since $\field=\End_{A*H_L}(L)=\End_A(L)^{H_L}$, we have $\End_A(L)=\field $ with the trivial $G$-action, so the above is isomorphic to
  \begin{align*}
    (\field \otimes \End_{\field}(V))^{H_L}
    \cong\End_{\field H_L}(V)
    \cong \field.
  \end{align*}
  Hence $\field G\otimes_{\field H_L}L\otimes V$ is indecomposable. Since it is semisimple by Corollary \ref{corollary_G/H}, it is thus simple.\\
  To see that these are up to isomorphism all simple $A*G$-modules, note that by Corollary \ref{corollary_dividessimples}, every simple $A*G$-module is a  summand of $\field G\otimes L$ for some $L\in \Sim(A)$ and
  \begin{align*}
    \field G \otimes L\cong \field G\otimes_{\field H_L}\field H_L \otimes L\cong \field G \otimes_{\field H_L} (L\otimes \field H_L).
  \end{align*}
  So decomposing $\field H_L$ into indecomposable summands yields the claim.\\
  Clearly, we have isomorphisms of $A*G$-modules
  \begin{align*}
    \field G\otimes_{\field H_L} L\otimes V\rightarrow \field G\otimes_{\field H_L} gL\otimes gV, h\otimes x\otimes v\mapsto hg\otimes x\otimes v.
  \end{align*}
  Finally, suppose we have an isomorphism of $A*G$-modules
  \begin{align*}
    \varphi:\field G\otimes_{\field H_L}(L\otimes V)\rightarrow \field G\otimes_{\field H_{L'}}(L'\otimes W).
  \end{align*}
  Then, restricting to $A$, we obtain an isomorphism
  \begin{align*}
    R_G(\varphi):R_G(\field G\otimes_{\field H_L}(L\otimes V))\rightarrow R_G(\field G\otimes_{\field H_{L'}}(L'\otimes W)).
  \end{align*} 
  Since 
  \begin{align*}
    &R_G(\field G\otimes_{\field H_L}(L\otimes V))\cong \bigoplus_{z\in Z_L}\dim_{\field}(V) zL\\
    \textup{ and }&R_G(\field G\otimes_{\field H_{L'}}(L'\otimes W))\cong \bigoplus_{z'\in Z_{L'}}\dim_{\field}(W) zL',
  \end{align*}
   the theorem of Krull-Remak-Schmidt thus yields a $g\in G$ such that $gL\cong L'$. In particular, $H_L=H_{L'}$ and, since $\field G\otimes_{\field H_L} L\otimes g^{-1}W\cong \field G\otimes_{\field H_L} gL\otimes W$ we have an isomorphism
  \begin{align*}
    \varphi':\field G\otimes_{H_L} (L\otimes V)\rightarrow \field G\otimes_{H_L} (L\otimes g^{-1}W)
  \end{align*}
  We can restrict this to $A*H_L$ to obtain an isomorphism
  \begin{align*}
    \bigoplus_{z\in Z_L}z(L\otimes V)\rightarrow \bigoplus_{z\in Z_L}z(L\otimes g^{-1}W).
  \end{align*}
  Since $L\otimes V$, $L\otimes g^{-1}W$ are simple $A*H_L$-modules by the above, we conclude that we have an isomorphism
  \begin{align*}
    \varphi'':L\otimes V\cong L\otimes g^{-1}W
  \end{align*}
  of $A*H_L$-modules.
  Now note that
  \begin{align*}
    \Hom_{A*H_L}(L\otimes V, L\otimes g^{-1}W) &\cong \Hom_{A}(L\otimes V, L\otimes g^{-1}W)^{H_L}\\
    \cong (\Hom_A(L, L)\otimes \Hom_{\field}(V, g^{-1}W))^{H_L}
    &\cong (\Hom_{\field}(V, g^{-1}W))^{H_L}
    \cong \Hom_{\field H_L}(V, g^{-1}W).
  \end{align*}
  Hence $gV\cong W$.
\end{proof}
The following is a counterexample to Lemma 2 in \cite{MartinezVilla}. There, it is erroneously claimed that if $S$ is a simple $A*G$ submodule of the socle of $A$, then $\Hom_A(S, S)\cong \field$.\\ 
This is in general false. However, it holds if, for example, all simple $\field G$-modules are one-dimensional, i.e. if $G$ is commutative. The reason for this is that in this case we have for any simple $A*G$-module $S$ and any irreducible representation $V$ of $G$ isomorphisms
\begin{align*}
  &\End_{A*G}(S\otimes V)\cong \End_{A}(S\otimes V)^G\cong (\End_A(S)\otimes \End_{\field}(V))^G\\
  &\cong   (\End_A(S)\otimes \field)^G\cong  \End_{A*G}(S)\cong \field
\end{align*}
where $G$ acts trivially on $\End_{\field}(V)$ as $\End_{\field}(V)\cong\field\cong \End_{\field G}(V)=\End_{\field}(V)^G$, so that $S\otimes V$ is irreducible.
Thus the result in \cite[Lemma 2]{MartinezVilla} holds in particular if $G$ is commutative.
\begin{example}\label{example_nonbasicsimples}
  Consider $\field=\mathbb{C}$, $A:=\field^5$, $G:=S_5$ acting on $A$ via permutations of the entries. $A$ is semisimple and basic, and $H_1\cong S_4$, so that by the above proposition we have a simple $A*G$-module $\field S_5\otimes_{\field S_{4}}(\field\times \{0\}^4\otimes V)$ for every irreducible representation $V$ of $S_4$. In particular, since $S_4$ has an irreducible representation $V$ of dimension 3, $A*G$ has a simple module of dimension 
  \begin{align*}
    \dim_{\field}(\field S_5\otimes_{\field S_{4}}(\field\times \{0\}^4\otimes V))=\dim_{\field}(\field S_5/S_{4})\dim_{\field}V=5\cdot 3=15.
  \end{align*}
  Moreover, note that $A$ is a simple $A*G$-module, and if $W$ is an irreducible representation of $S_5$, 
  \begin{align*}
    \dim_{\field}(A\otimes W)=5\dim_{\field}W.
  \end{align*}
  Since $S_5$ has no irreducible representations of dimension 3, this implies that not all simple $A*G$-modules are of the form $A\otimes W$.
\end{example}
\begin{corollary}\label{corollary_indecproj_basic}
The indecomposable projective $A*G$-modules are exactly of the form $$\field G\otimes_{\field H_L}(P_L\otimes V)$$ for some irreducible $\field H_L$-module $V$. They are isomorphic if and only if there is $g\in G$ such that $gL=L', gV=W$.
\end{corollary}
\begin{proof}
  Since $V|\field H_L$, $\field G\otimes_{\field H_L}(P_L\otimes V)$ is a direct summand of $\field G\otimes P_L$ as above, and is therefore projective. Moreover, by Lemma \ref{lemma_radical} and Proposition \ref{prop_IR+RI}
  \begin{align*}
    \topp(\field G\otimes_{\field H_L}(P_L\otimes V))&=\field G\otimes_{\field H_L}\topp(P_L\otimes V)\\
    &=\field G\otimes_{\field H_L}(\topp(P_L)\otimes V)\\
    &=\field G\otimes_{\field H_L}(L\otimes V).
  \end{align*}
  By Proposition \ref{prop_simplesofA*G}, this is simple, so that $\field G\otimes_{\field H_L}(P_L\otimes V)$ is indecomposable. 
\end{proof}
\section{Quasi-Hereditary Algebras}\label{s2}
In this section, we shall repeat some standard definitions and results about \puhh algebras, as introduced by \cite{Scott} and \cite{CPS}. For an introduction to quasi-hereditary algebras see for example \cite{DlabRingel}. \\
Let $A$ be an algebra. Denote by $\Sim(A)$ the set of simple $A$-modules and suppose $\leq$ is a partial order on $\Sim(A)$.
\begin{remark}\label{remark_extendorder}
  Suppose $\leq$ is a partial order on the set $\Sim(A)$ of simple $A$-modules. Then this induces a partial order on its additive closure $\add(\Sim(A))$ via
  \begin{align*}
    S\leq S'\Leftrightarrow L\leq L' \textup{ for all }  L|S, L'|S'.
  \end{align*}
  Thus, if $\leq$ is a partial order on $\Sim(A)$, we will also use it to compare semisimple modules.
\end{remark}
\begin{definition}\label{def_adapted}\cite[p. 3]{DlabRingel}
  We call a partial order $\leq$ on $\Sim(A)$ adapted, if all $M\in \modu A$ with simple top $\Top(M)=L$ and simple socle $\soc(M)=L'$, such that $L$ and $L'$ are incomparable with respect to $\leq$, have a composition factor $L''$ such that $L''>L$ and $L''>L'$.
  \end{definition}
\begin{lemma}\label{lemma_compfactors} 
\begin{enumerate}
  \item Let $M$ be a module with a composition factor $L'$. Then there is a factor module $M'$ of $M$ with socle $\soc(M')\cong L'$.
  \item Let $M$ be a module with a composition factor $L'$. Then there is a submodule $M'$ of $M$ with top $\topp(M')\cong L'$.
\end{enumerate}
\end{lemma}
\begin{proof}
  \begin{enumerate}
    \item By definition of a composition factor, there is a factor module $M''$ of $M$ such that we have an embedding $\iota:L'\rightarrow M''$.\\
    Now let $N$ be a maximal submodule of $M''$ subject to $N\cap \im(\iota)=(0)$. Then $M':=M''/N$ is also a factor module of $M$. Denote by $$\pi:M''\rightarrow M'$$ the canonical projection. Then $\pi\circ \iota:L'\rightarrow M'$ is injective, so that $L'|\soc(M')$. Write $\soc(M')=\pi\circ\iota(L')\oplus S$. Then $N\subseteq \pi^{-1}(S)$ and $\iota(L')\cap \pi^{-1}(S)=(0)$. Hence by maximality of $N$, $\pi^{-1}(S)=N$, so that $S=(0)$.
    \item This is dual to 1.
  \end{enumerate}
  \end{proof}
\begin{lemma}\label{lemma_adaptedrewrite}
  The partial order $\leq$ is adapted if and only if every module $M$ which has a composition factor $L'$ such that $L'$ is incomparable to every summand $L$ of its top $\topp(M)$ has a composition factor $L''$ and a composition factor $L|\topp(M)$ which is a summand of the top, such that $L''>L$.
 \end{lemma}
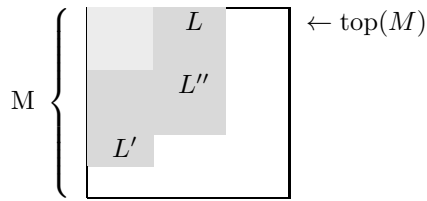
 \begin{center}
\begin{tabular}{c|c c c c|c}
   \cline{2-5}
     \multirow{6}{*}{M} \ldelim\{{6}{*} &\cellcolor{gray!15} & \cellcolor{gray!30}  $L$ &  & & $\leftarrow \Top(M)$ \\
     & \cellcolor{gray!15}& \cellcolor{gray!30} &  & & \\
     & \cellcolor{gray!30}& \cellcolor{gray!30}  $L''$ &   & & \\
     & \cellcolor{gray!30}& \cellcolor{gray!30} &    & & \\
     & \cellcolor{gray!30} $L'$ & &  &  & \\ 
       &  & &  & \\
       \cline{2-5}
  \end{tabular}
\end{center}
\begin{proof}
  Suppose $\leq$ is adapted and let $M$ be a module with a composition factor $L'$ such that $L'$ is incomparable to every summand $L$ of its top $\topp(M)$.\\
  By Lemma \ref{lemma_compfactors}, $M$ has a factor module $M'$ with simple socle $L'$.
   Let $L$ be any summand of $\topp(M')$. Then, since $M'$ is a factor module of $M$, we have $\topp(M')|\topp(M)$. Hence $L$ is also a summand of $\topp(M)$, and thus in particular incomparable to $L'$.
    Moreover, we can again apply Lemma \ref{lemma_compfactors} to obtain a submodule $M''$ of $M'$ with simple top $\topp(M'')\cong L$. As $M''$ is a submodule of $M'$, $\soc(M'')|\soc(M')\cong L'$, so that $M''$ has simple socle isomorphic to $L'$. Since $L$ and $L'$ are incomparable, $M''$ thus has a composition factor $L''>L, L'$, and since $L''$ is a composition factor of $M''$, which is a submodule of a factor module of $M$, $L''$ is also a composition factor of $M$. Hence this proves the first implication.\\
  On the other hand, suppose every module $M$ which has a composition factor $L'$ such that $L'$ is incomparable to every summand of its top has a composition factor $L''$ and a composition factor $L|\topp(M)$ such that $L''>L$, and let $M$ be a module with simple top $L$ and simple socle $L'$ such that $L$ and $L'$ are incomparable.\\
   Then by assumption, $M$ has a composition factor $L''$ such that $L''>L$. Without loss of generality we can choose $L''$ maximal with respect to $L''>L$. Then by Lemma \ref{lemma_compfactors}, $M$ has a submodule $M'$ with simple top $\topp(M')\cong L''$. Since $M$ has simple socle $L'$, $M'$ also has simple socle $L'$. Now if $L'$ and $L''$ were incomparable, then by assumption $M'$ would have a composition factor $L'''>L''$, which is a contradiction to the maximality of $L''$. Hence $L'$ and $L''$ are comparable. Since $L''>L$ and $L'$ and $L$ are incomparable, this implies $L'<L''$.
\end{proof}
\begin{definition}\label{def_standardmodule}
  Let $\leq$ be a partial order on $\Sim(A)$. Then for every simple $A$-module $L$ we define 
  \begin{align*}
    \Delta_L:=P_L/\left(\sum_{L'\nleq L, \varphi\in \Hom_A(P_{L'}, P_L)}\im(\varphi)\right)
  \end{align*}
  and 
  \begin{align*}
    \fakedelta_L:=P_L/\left(\sum_{L'> L, \varphi\in \Hom_A(P_{L'}, P_L)}\im(\varphi)\right).
  \end{align*}
  Denote by $\pi_L:P_L\rightarrow \Delta_L$ and $\fakepi_L:P_L\rightarrow \fakedelta_L'$ the canonical projection. 
  Moreover, write $$\Delta:=\bigoplus_{L\in \Sim(A)}\Delta_L$$ and $$\fakedelta:=\bigoplus_{L\in \Sim(A)}\fakedelta_L$$
   and call $(\Delta_L)_{L\in \Sim(A)}$ the collection of standard modules and $(\fakedelta_L)_{L\in \Sim(A)}$ the collection of pseudostandard modules for $(A, \leq)$.
  \end{definition}
  Later, in case more than one algebra is involved, we will sometimes add a superscript to $\Delta$, $\fakedelta$, $\Delta_L$ and $\fakedelta_L$ indicating the respective algebra.
  \begin{lemma}\label{lemma_standardmodules}
    Let $L\in \Sim(A)$. Then the following statements hold:
    \begin{enumerate}
      \item $L'\nleq L$ for every summand $L'$ of $\Top(\ker(\pi_L))$ and for every epimorphism $f:P_L\rightarrow M$ such that $L'\nleq L$ for every summand $L'$ of $\Top(\ker(f))$ we have an epimorphism $g:M\rightarrow \Delta_L$ such that $\pi_L=g\circ f$.
      \item $L' \leq L$ for every composition factor $L'$ of $\Delta_L$ and for every homomorphism $f:P_L\rightarrow M$ such that $L'\leq L$ for every composition factor $L'$ of $M$ we have a homomorphism $g:\Delta_L\rightarrow M$ such that $f=g\circ \pi_L$.
      \item We have $\Top(\ker(\fakepi_L))>L$ and for every epimorphism $f:P_L\rightarrow M$ such that $\Top(\ker(f))>L$ we have an epimorphism  $g:M\rightarrow \fakedelta_L$ such that $\fakepi_L=g\circ f$.
      \item $L' \ngtr L$ for every composition factor $L'$ of $\fakedelta_L'$ and for every homomorphism $f:P_L\rightarrow M$ such that $L'\ngtr L$ for every composition factor $L'$ of $M$ we have a homomorphism $g:\fakedelta_L'\rightarrow M$ such that $f=g\circ \fakepi_L$.
    \end{enumerate}
  \end{lemma}
  \begin{proof}
    \begin{enumerate}
      \item We have $L'\nleq L$ for every summand $L'$ of $\Top(\ker(\pi_L))$ by definition. So let $f:P_L\rightarrow M$ be an epimorphism such that $L'\nleq L$ for every summand $L'$ of $\Top(\ker(f))$. Then we have a projection $\pi:\bigoplus_{L'\nleq L}n_{L'}P_{L'}\rightarrow \ker(f)$ for some $n_{L'}\in \mathbb{N}_0$. Composing with the embedding yields that $$\ker(f)\subseteq \sum_{L'\nleq L, \varphi\in \Hom_A(P_{L'}, P_L)}\im(\varphi)=\ker(\pi_L).$$ Hence $\pi_L$ factors through $f$.
      \item By definition, $L'\leq  L$ for every composition factor $L'$ of $\Delta_L$. So let $f:P_L\rightarrow M$ such that $L'\nleq L$ for every composition factor $L'$ of $M$. Let $L'\nleq L$ and let $\varphi:P_{L'}\rightarrow P_L$. Then, since all composition factors of $M$  are less than or equal to $L$, $f\circ \varphi=0$. Hence $\im(\varphi)\subseteq \ker(f)$, so that $\ker(\pi_L)\subseteq \ker(f)$ and thus $f$ factors through $\pi_L$.
      \item This is analogous to 1.
      \item This is analogous to 2.
    \end{enumerate}
  \end{proof}
  \begin{lemma} \label{lemma_extdelta}
    For every $L, L'\in \Sim(A)$ such that $\Ext^1(\fakedelta_{L'}, \Delta_{L})\neq (0)$ we have $L'<L$
  \end{lemma}
  \begin{proof}
    By definition, the module $\fakedelta_{L'}$ has a projective presentation
\[\begin{tikzcd}[ampersand replacement=\&]
	{\bigoplus_{L''>L'}n_{L''}P_{L''}} \& {P_L'} \& {\fakedelta_L'} \& {(0)}
	\arrow[from=1-1, to=1-2]
	\arrow[from=1-2, to=1-3]
	\arrow[from=1-3, to=1-4]
\end{tikzcd}\]
for some integers $n_{L''}\in \mathbb{N}_0$. Suppose $\Ext^1_A(\fakedelta_{L'}, \Delta_{L})\neq (0)$. Then 
\begin{align*}
  \Hom_A({\bigoplus_{L''>L'}n_{L''}P_{L''}}, \Delta_L)\neq (0),
\end{align*}
so that for some $L''>L'$ 
\begin{align*}
  \Hom_A({P_{L''}}, \Delta_L)\neq (0).
\end{align*}
Thus $L''$ is a composition factor of $\Delta_L$. Now since every composition factor of $\Delta_{L}$ is less than or equal to $L$, this implies $L'<L''\leq L$.
  \end{proof}
  The following definition is due to \cite{DlabRingel}; it resembles the definition of an exceptional collection, originating in the work of Beilinson \cite{Beilinson1, Beilinson2}, developed in \cite{Rudakov} and \cite{Bondal}, the only difference being that condition 3. is here only required for $\Ext^1$ instead of for $\Ext^n$ for all $n\geq 1$.
\begin{definition}\label{def_standardisableset}
  Let $\leq$ be a partial order on $\Sim(A)$. Then a standardisable set for $(A, \leq )$ is a family of modules $M=(M_L)_{L\in \Sim(A)}$ such that
  \begin{enumerate}
    \item $\Top(M_L)\cong L$.
    \item $\Hom_A(M_L, M_{L'})\neq (0)\Rightarrow L\leq  L'$.
    \item $\Ext_A^1(M_L, M_{L'})\neq (0)\Rightarrow L< L'$.
  \end{enumerate}
\end{definition}
\begin{remark}\label{remark_standardisable}
  Note that $(\Delta_L)_L$ fulfills condition 1. and 2. in the above definition. Thus by Lemma \ref{lemma_extdelta}, if $\fakedelta_L=\Delta_L$ for every $L\in \Sim(A)$, then $(\Delta_L)_L=(\fakedelta_{L})_L$ is a standardisable set.
\end{remark}
The following lemma tells us that $(\Delta_L)_L=(\fakedelta_{L})_L$ if and only if $\leq$ is adapted.
Moreover, the former is the case if and only if any refinement of our partial order will give rise to the same set of standard modules. Thus, being adapted means that our partial order is, in a sense, fine enough.
\begin{lemma}\label{lemma_adapted}
  The following statements are equivalent:
  \begin{enumerate}
    \item $\Delta_L=\fakedelta_L$ for every $L\in \Sim(A)$.
    \item $\leq$ is adapted.
    \item $(\fakedelta_L)_L$ is a standardisable set.
    \item $\Hom_A(\fakedelta_{L'}, \fakedelta_L)\neq (0) \Rightarrow L'\leq L$.
  \end{enumerate}
\end{lemma}
\begin{proof}
  \begin{enumerate}
    \item[$1\Rightarrow 2$] Suppose $\Delta_L=\fakedelta_L$ for every $L\in \Sim(A)$ and let $M$ be an $A$-module with simple top $L$ and socle $L'$. Suppose no composition factor $L''$ of $M$ is bigger than $L$. Since $L=\top(M)$ there is an epimorphism $\pi_M:P_L\rightarrow M$. Now since no composition factor $L''$ of $M$ is bigger than $L$, Lemma \ref{lemma_standardmodules} yields a homomorphism $g:\fakedelta_L\rightarrow M$ such that $\pi_M=g\circ \fakepi_L$. In particular, $g$ is an epimorphism, so that, since every composition factor of $\fakedelta_L=\Delta_L$ is less than or equal to $L$, every composition factor of $M$ is less than or equal to $L$. In particular $L'$ and $L$ are comparable.
    \item[$2\Rightarrow 3$] This holds by Remark \ref{remark_standardisable}.
    \item[$3\Rightarrow 4$] This holds by definition.
    \item[$4\Rightarrow 1$] Suppose there is some $L\in \Sim(A)$ such that $\Delta_L\neq \fakedelta_L$. Then $\fakedelta_L$ has some composition factor $L'\nleq L$. Let $L'$ be a maximal such composition factor. Then there is a non-zero homomorphism $f:P_{L'}\rightarrow \fakedelta_L$. Moreover, since $L'$ is a maximal composition factor of $\fakedelta_L$, every composition factor $L''\ngtr L'$ for every composition factor $L''$ of $\fakedelta_L$. Thus Lemma \ref{lemma_standardmodules} yields a homomorphism $g: \fakedelta_{L'}\rightarrow \fakedelta_L$ such that $f=g\circ \fakepi_{L'}$. In particular, 
        $$\Hom_A(\fakedelta_{L'}, \fakedelta_L)\neq (0).$$
    \end{enumerate}
\end{proof}
\begin{example}
  Consider the algebra $A$ given by the quiver 
\[\begin{tikzcd}[ampersand replacement=\&]
	a \& b \& c
	\arrow["\beta", from=1-2, to=1-3]
	\arrow["\alpha", from=1-1, to=1-2]
\end{tikzcd}\]
    with relations $\langle \beta\alpha \rangle=J^2$ and the partial order on $\Sim(A)=\{L_a, L_b, L_c\}$ given by $L_c<L_a$ and $L_c<L_b$.
    Then the indecomposable projective modules are given by
    \begin{align*}
      P_a=\begin{pmatrix} a \\ b\end{pmatrix}\textup{, }
      P_b=\begin{pmatrix} b \\ c\end{pmatrix}
      \textup{ and } P_c=(c),
    \end{align*}
    so $\fakedelta_i=P_i$ for every $i\in \{a, b, c\}$. In particular, $\Hom_A(\fakedelta_b, \fakedelta_a)\neq (0)$, so that  $(\fakedelta_i)_i$, is not standardisable. \\
    On the other hand $\Delta_a=L_a$, $\Delta_b=P_b$ and $\Delta_c=P_c$, so that
    \begin{align*}
      \Ext^1_A(\Delta_j, \Delta_i)=(0)
    \end{align*}
    for all $n\geq 1$, $i\in \{a,b,c\}$ and $j\in \{b,c\}$.
    Moreover
\[\begin{tikzcd}[ampersand replacement=\&]
	{(0)} \& {P_c} \& {P_b} \& {P_a} \& {L_a} \& {(0)}
	\arrow[from=1-1, to=1-2]
	\arrow["{r_\beta}", from=1-2, to=1-3]
	\arrow["{r_\alpha}", from=1-3, to=1-4]
	\arrow["{\pi_a}", from=1-4, to=1-5]
	\arrow[from=1-5, to=1-6]
\end{tikzcd}\]
is a projective resolution of $L_a$, where $r_\alpha$ and $r_\beta$ denote right multiplication by $\alpha$ resp. $\beta$ and $\pi_a$ is the canonical projection. Since $\Hom_A(P_b, \Delta_a)=\Hom_A(P_b, \Delta_c)=(0)$, this implies that
\begin{align*}
  \Ext^1_A(\Delta_a, \Delta_a)=(0)=\Ext^1_A(\Delta_a, \Delta_c).
\end{align*}
  Finally, since 
  \begin{align*}
    r_\beta^*:\Hom_A(P_b, \Delta_b)\rightarrow \Hom_A(P_c, \Delta_b)
  \end{align*}
  is injective, we also obtain that $\Ext^1_A(\Delta_a, \Delta_b)=(0)$. Thus $(\Delta_i)_i$ is standardisable.
\end{example}
The following definition is an adaptation of the definition of an exceptional sheaf which can be found in \cite{Rudakov}, where, as in definition \ref{def_standardisableset}, we replace the requirement on $\Ext^n$ for $n\geq 0$ by a requirement only on $\Ext^1$. 
\begin{definition}\label{def_exceptional}
  An $A$-module $M$ is called exceptional if $\Ext_A^1(M, M)=(0)$ and $\End_A(M)\cong \field$.
\end{definition}
\begin{lemma}\label{lemma_adaptedandexceptionalrewrite}
  The following statements are equivalent:
  \begin{enumerate}
    \item $\leq$ is adapted and $\Delta_L=\fakedelta_L$ is exceptional for every $L\in \Sim(A)$.
    \item Every composition factor $L'$ of $\rad(\fakedelta_{L})$ fulfills $L'<L$.
    \item $\Hom_A(\fakedelta_{L'}, \rad(\fakedelta_{L}))\neq (0)\Rightarrow L'<L$.
  \end{enumerate}
\end{lemma}
\begin{proof}
  \begin{enumerate}
    \item[$1\Rightarrow 2$] Recall that if $\leq$ is adapted, then $(\Delta_L)_L=(\fakedelta_L)_L$. Let $L'$ be a composition factor of $\rad(\fakedelta_{L})=\rad(\Delta_L)$. Then $L'\leq L$. If $L=L'$ then this induces a non-zero homomorphism $P_{L}\rightarrow \rad(\Delta_{L})$ and thus, by Lemma \ref{lemma_standardmodules}, an endomorphism $\Delta_{L}\rightarrow \rad(\Delta_{L})\rightarrow \Delta_{L}$ which is neither zero nor invertible. This is a contradiction. Thus $L'<L$.
    \item[$2\Rightarrow 3$] Suppose $\Hom_A(\fakedelta_{L'}, \rad(\fakedelta_{L}))\neq (0)$. Then, since $\fakedelta_{L'}$ has simple top $L'$, $L'$ is a composition factor of $\rad(\fakedelta_{L})$ and thus $L'<L$.
    \item[$3\Rightarrow 1$] 
    If $L\neq L'$ then  $$\Hom_A(\fakedelta_{L'}, \rad(\fakedelta_{L}))= (0)\Leftrightarrow \Hom_A(\fakedelta_{L'}, \fakedelta_{L})=(0).$$ Thus $$\Hom_A(\fakedelta_{L'}, \fakedelta_{L})\neq (0)\Rightarrow L'\leq L,$$ so that $\leq$ is adapted,  $\fakedelta_L=\Delta_L$ and $(\Delta_L)_L$ is a standardisable set by Lemma \ref{lemma_adapted}.\\
    In particular, $\Ext^1(\Delta_L, \Delta_L)=(0).$\\
    Moreover, $\Hom_A(\fakedelta_L, \rad(\fakedelta_L))=(0)$, so that any endomorphism of $\fakedelta_L$ is either surjective or zero.
    Thus $\End_A(\Delta_L)=\End_A(\fakedelta_L)\cong \field.$ \qedhere
  \end{enumerate}
\end{proof}
\begin{definition}
  We denote by $\mathcal{F}(\Delta)$ the full subcategory of $\modu A$ which contains all $A$-modules admitting a filtration by the $\Delta_L$, $L\in \Sim(A)$. In other words, an $A$-module $M$ is in $\mathcal{F}(\Delta)$ if and only if there is an integer $m\geq 0$ and an ascending sequence of submodules 
  \begin{align*}
    (0)=M_0\subset M_1 \subset \dots \subset M_m=M
  \end{align*}
  such that for every $1\leq i\leq m$ there is an $L_i\in \Sim(A)$ such that $M_i/M_{i-1}\cong \Delta_{L_i}$.
\end{definition}
\begin{proposition}\label{prop_addclosed}\cite[Lemma 1.4 and Lemma 1.5]{DlabRingel}
  The subcategory $\mathcal{F}(\Delta)$ is closed under direct sums, direct summands and extensions.
\end{proposition}
\begin{definition} \label{def_qh}\cite{CPS, Ringel, stratified}
  An algebra $A$ together with an adapted partial order $\leq$ on $\Sim(A)$ is called 
  \begin{enumerate}
    \item left standardly stratified, if $A\in \mathcal{F}(\Delta)$
    \item quasi-hereditary, if additionally $\Delta_L$ is exceptional for all $L\in \Sim(A)$
    \item strongly quasi-hereditary, if additionally every $\Delta_L$ has projective dimension one.
    \item directed, if $\Delta_L\cong L$ for all $L\in \Sim(A)$.
  \end{enumerate}
\end{definition}
\begin{definition}\label{def_borel}\cite[p. 405]{Koenig}\cite[Definition 3.4]{BKK}
  Let $(A, \leq)$ be a quasi-hereditary algebra. Then a subalgebra $B\subseteq A$ is called an exact Borel subalgebra if
  there is a bijection $i:\Sim(B)\rightarrow \Sim(A)$ such that
  \begin{enumerate}
    \item $A$ is projective as a right $B$-module,
    \item $B$ is directed
    \item $A\otimes_B L=\Delta_{i(L)}$ for all $L\in \Sim(B)$.
  \end{enumerate}
  Moreover, it is called
  \begin{enumerate}
    \item a strong exact Borel subalgebra if there is a maximal semisimple subalgebra of $A$ which is also a semisimple subalgebra of $B$;
    \item a homological exact Borel subalgebra if the induced maps
    \begin{align*}
      A\otimes_B- :\Ext_B^*(L, L')\rightarrow \Ext_A^*(\Delta_L, \Delta_{L'})
    \end{align*}
    are isomorphisms in degree greater than or equal to two and epimorphisms in degree one for all $L, L'\in \Sim(B)$;
    \item a normal exact Borel subalgebra, if there is a splitting of the inclusion $\iota:B\rightarrow A$ whose kernel is a right ideal in $A$;
    \item and a regular exact Borel subalgebra if it is normal and the induced maps
    \begin{align*}
      A\otimes_B- :\Ext_B^*(L, L')\rightarrow \Ext_A^*(\Delta_{i(L)}, \Delta_{i(L')})
    \end{align*}
    are isomorphisms in degree greater than or equal to one for all $L, L'\in \Sim(B)$.
  \end{enumerate}
\end{definition}
\begin{lemma}\label{lemma_strongBorel}
  Suppose $B$ is an exact Borel subalgebra of $A$. Then $B$ is a strong exact Borel subalgebra of $A$ if and only if $A\rad(B)\subseteq \rad(A)$.
\end{lemma}
\begin{proof}
  Let $L^A:=A/\rad(A)$ and $L^B:=B/\rad(B)$. Then $L^A$ and $L^B$ are up to isomorphism the unique maximal semisimple subalgebras of $A$ and $B$ respectively.\\
  In particular, a semisimple subalgebra of $A$ is a maximal semisimple subalgebra if and only if has the same vector space dimension as $L^A$.
  Since any semisimple subalgebra of $B$ is also a semisimple subalgebra of $A$, this means that $B$ contains a maximal semisimple subalgebra of $A$ if and only if $\dim_{\field}L^B=\dim_{\field}L^A$.\\
  Thus $B$ is a strong exact Borel subalgebra of $A$ if and only if $\dim_{\field}L^B=\dim_{\field}L^A$.
  Let $X:=A\rad(B)/(\rad(A)\cap A\rad(B))$.
Then, as $A$-modules, 
\begin{align*}
  L^A= A/\rad(A) &\cong ((\rad(A)+A\rad(B))/\rad(A))\oplus A/(\rad(A)+A\rad(B))\\
  &\cong A\rad(B)/(\rad(A)\cap A\rad(B))\oplus \topp(A/A\rad(B))\\
  &\cong X\oplus  \topp\left(A\otimes_B B/\rad(B)\right)\\
  &\cong X\oplus  \topp\left(\bigoplus_{i}[\topp B: L_i^B]\Delta_{L_i^A}\right)\\
  &\cong X\oplus \bigoplus_{i}[L^B: L_i^B]L_i^A.
\end{align*}
Thus, 
\begin{align*}
  \dim_{\field}L_i^A&=[L^A:L_i^A]=[X:L_i^A]+[L^B: L_i^B]=[X:L_i^A]+\dim_{\field}L_i^B,\end{align*}
 so that
\begin{align*}
  \dim_{\field} L^A&=\dim_{\field} X+\sum_i[L^B: L_i^B]\dim_{\field} L_i^A\\
  &= \dim_{\field} X+\sum_i[L^B: L_i^B]([X:L_i^A]+\dim_{\field}(L_i^B))\\
  &=\dim_{\field} X+\sum_{i}[X:L_i^A][L^B:L_i^B]+ \dim_{\field}L^B.
\end{align*}
Hence $B$ is a strong exact Borel subalgebra of $A$ if and only if $X=(0)$, i.e. if $A\rad(B)\subseteq B$.
\end{proof}
\section{Quasi-Hereditary Algebras with a Group Action}\label{s3}
Throughout, let, as before, $A$ be a finite-dimensional $\field$-algebra and $G$ be a group acting on $A$ via automorphisms such that the order of $G$ does not divide the characteristic of $\field$. 
\begin{definition}\label{def_Gequivorder}
  A partial order $\leq_A$ on $\Sim(A)$ is called $G$-equivariant if
  \begin{align*}
    L<_AL' \Leftrightarrow gL<_AhL' \textup{ for all } g,h\in G.
  \end{align*}
  On the other hand, a partial order $\leq_{A*G}$ on $\Sim(A*G)$ is called $G$-stable if
  \begin{align*}
    S<_{A*G}S' \Leftrightarrow S\otimes V<_{A*G} S' \otimes W
  \end{align*}
  for all $\field G$-modules $V, W$.
  \end{definition}
  
\begin{definition}\label{def_inducedorder}
  \begin{enumerate}
    \item Let $\leq$ be  a $G$-equivariant partial order on $\Sim(A)$. Then we define a partial order $\leq_{G}$ on $\Sim(A*G)$ via
    \begin{align*}
      S<_GS':\Leftrightarrow _{A|}S<_{|A}S'.
    \end{align*}
    \item Let $\leq'$ be  a $G$-stable partial order on $\Sim(A*G)$. Then we define a partial order $\leq'_{|G}$ on $\Sim(A)$ via
    \begin{align*}
      L<'_{|G}L':\Leftrightarrow \field G\otimes L <'\field G\otimes L'.
    \end{align*}
  \end{enumerate}
\end{definition}
Hence we define a strict partial order $(<_{A})_G$ as the pullback of a strict partial order $<_A$ along the map
\begin{align*}
\add{\Sim(A*G)}\rightarrow \add(\Sim(A)), M\mapsto {}_{|A}M,
\end{align*}
and similarly 
a strict partial order $(<_{A*G})_{|G}$ as the pullback of a strict partial order $<_{A*G}$ along the map
\begin{align*}
\add{\Sim(A)}\rightarrow \add(\Sim(A*G)), M\mapsto I_GM.
\end{align*}
Note that since the above maps are not necessarily injective, the pullbacks of $\leq_A$ resp. $\leq_{A*G}$ along these maps are not necessarily partial orders, since asymmetry does in general not hold.\\
However, since the pullback of a strict partial order is always a strict partial order, the above yield well-defined partial orders $\leq_G$ and $\leq'_{|G}$  even for a not necessarily $G$-equivariant partial order $\leq$ and a not necessarily $G$-stable partial order $\leq'$.\\
However, if e.g. $\leq$ is not $G$-equivariant, then there may be some simple $A*G$ modules $S\neq S'$ such that $L<_AL'$, but $L''\nless L'''$ for some simple summands of $L, L''$ of $S$ and $L', L'''$ of $S'$, so $S$ and $S'$ would be incomparable with respect to $\leq_A$, even though all summands of ${}_{|A}S$ and ${}_{|A}S'$ may be comparable. Hence even a total order $\leq$ might result in an empty order $\leq_G$, and so there is little hope to conclude adaptedness of $\leq_G$ from adaptedness of $\leq$.\\
Similar considerations hold for a non $G$-stable partial order $\leq'$ and its induced partial order $<'_{|G}$.
\begin{proposition}\label{prop_bijectionorders}
  \begin{enumerate}
    \item Let $\leq$ be a $G$-equivariant partial order on $\Sim(A)$. Then $\leq_{G}$ is the unique $G$-stable partial order such that
    \begin{align*}
      L< L'\Leftrightarrow \field G\otimes  L<_{G}\field G\otimes L'.
    \end{align*}
    \item Let $\leq'$ be a $G$-stable partial order on $\Sim(A*G)$. Then $\leq'_{|G}$ is the unique $G$-equivariant partial order on $\Sim(A)$ such that 
    \begin{align*}
      S<'S'\Leftrightarrow _{A|}S(<'_{|G})_{A|}S'.
    \end{align*}
  \end{enumerate}
\end{proposition}
\begin{proof}
  \begin{enumerate}
    \item Clearly, $\leq_G$ is $G$-stable, and 
    \begin{align*}
      L<L'\Leftrightarrow \field G\otimes  L<_{G}\field G\otimes L'.
    \end{align*}
    Now suppose that $\leq'$ is another $G$-stable partial order on $\Sim(A*G)$ such that 
    \begin{align*}
      L<L'\Leftrightarrow \field G\otimes  L<'\field G\otimes L'.
    \end{align*}
    Let $S, S'\in \Sim(A*G)$.
    Then there are $L, L'\in \Sim(A)$ such that $L|_{A|}S$, $L'|_{A|}S'$.
    Thus $\field G\otimes L| \field G\otimes {}_{A|}S\cong S\otimes \field G$ and $\field G\otimes L'| \field G\otimes {}_{A|}S'\cong S'\otimes \field G$.\\
    Moreover, by Corollary \ref{corollary_dividessimples}, there are $L'', L'''\in \Sim(A)$ such that $S|\field G\otimes L'', S'|\field G\otimes L'''$.
    In particular, $L|{}_{A|}\field G\otimes L''$ and $L'|{}_{A|}\field G\otimes L'''$,
    so that $L\cong gL''$  and $L'\cong hL'''$ for some $g, h\in G$. Hence $\field G\otimes L\cong \field G\otimes L''$ and $\field G\otimes L'\cong \field G\otimes L'''$.
    Thus \begin{align*}
      S<'S'\Leftrightarrow S\otimes \field G<'S'\otimes \field G
      &\Rightarrow \field G\otimes L<'\field G\otimes L'\\
      &\Leftrightarrow L<L'
      \Leftrightarrow \field G\otimes L<_G\field G\otimes L'
      \Rightarrow S<_{G}S',
    \end{align*}
    and analogously $S<_GS'\Rightarrow S<' S'.$
\item Clearly, $<'_{|G}$ is $G$-equivariant and 
\begin{align*}
  {}_{A|}S<'_{|G}{}_{A|}S'&\Leftrightarrow \field G\otimes _{A|}S<'\field G\otimes _{A|}S'\\
  &\Leftrightarrow S\otimes \field G<'S'\otimes \field G\Leftrightarrow S<S'.
\end{align*}
Moreover, if $\leq$ is another $G$-equivariant partial order on $\Sim(A)$ such that 
\begin{align*}
  S<'S'\Leftrightarrow _{A|}S<_{A|}S'.
\end{align*}
then 
\begin{align*}
  L<L'\Leftrightarrow \bigoplus_{g\in G}gL<\bigoplus_{g\in G}gL'
  \Leftrightarrow _{A|}\field G\otimes L<_{A|}\field G\otimes L'
  \Leftrightarrow \field G\otimes L<'\field G\otimes L'
  \Leftrightarrow L<'_G L'
\end{align*}
for all $L, L'\in \Sim(A)$.
  \end{enumerate}
\end{proof}
\begin{corollary}\label{corollary_bijectionorders}
  Let $\leq$ be a $G$-equivariant partial order on $\Sim(A)$. Then $\leq$ coincides with $(\leq_G)_{|G}$.\\
  Let $\leq'$ be a $G$-stable partial order on $\Sim(A*G)$. Then $\leq'$ coincides with $(<'_{|G})_{G}$.
  Thus the assignments
  \begin{align*}
    \{G\textup{-equivariant partial} & \textup{ orders on }\Sim(A)\}\\
    &\rightarrow \{G\textup{-stable partial orders on }\Sim(A*G)\},\\ 
    \leq &\mapsto \leq_G
  \end{align*}
  and
  \begin{align*}
    \{G\textup{-stable partial orders on }&\Sim(A*G)\}\\ &\rightarrow \{G\textup{-equivariant partial orders on }\Sim(A)\},\\ \leq' &\mapsto \leq'_{|G}
  \end{align*}
  are mutually inverse bijections.
\end{corollary}
From now on, let $\leq_A$ be a $G$-equivariant partial order on $\Sim(A)$ and $\leq_{A*G}$ the corresponding induced order on $\Sim(A*G)$, or the other way around.
\begin{lemma}\label{lemma_gdelta=deltag}
  For every $g\in G$, $L\in \Sim(A)$ we have $g\Delta_{L}\cong \Delta_{g(L)}$ and $g\fakedelta_{L}\cong \fakedelta_{g(L)}$
\end{lemma}
\begin{proof}
  This is a direct consequence of the fact that $g$ induces an order preserving automorphism of $\modu A$.
\end{proof}
\begin{proposition}\label{prop_IDelta}
  For every simple $A$-module $L\in \Sim(A)$.
  \begin{align*}
    \field G\otimes \fakedelta_{L}\cong\bigoplus_{S\in \Sim(A*G)}[\field G\otimes L:S]\fakedelta_{S}
  \end{align*}
  where $[\field G\otimes L:S]$ denotes the multiplicity of the simple summand $S$ in $\field G\otimes L$.
\end{proposition}
\begin{proof}
  Since $\field G\otimes P_{L}$ is projective with top $$\Top(\field G\otimes P_{L})\cong \field G\otimes \Top(P_{L})=\field G\otimes L$$ by Lemma \ref{lemma_radical} and Proposition \ref{prop_IR+RI}, there is an isomorphism
  \begin{align*}
    \varphi_L': \bigoplus_{S\in \Sim(A*G)}[\field G\otimes L:S]P_{S}\rightarrow \field G\otimes P_{L}.
  \end{align*}
For the sake of notation, fix a set $\mathcal{S}_L$ of simple $A*G$-modules such that for any $S\in \Sim(A*G)$ there are exactly $[\field G\otimes L:S]$ modules in $\mathcal{S}_L$ which are isomorphic to $S$ and consider instead the isomorphism
\begin{align*}
    \varphi_L: \bigoplus_{S\in \mathcal{S}_L}P_{S}\rightarrow \field G\otimes P_{L}.
  \end{align*}
  For every $S\in \mathcal{S}_L$ consider the map
  \begin{align*}
    g_S:=(\field G\otimes \fakepi_L)\circ \varphi_L\circ \iota_{P_{S}}:P_{S}\rightarrow \field G\otimes \fakedelta_{L}
  \end{align*}
  where $\iota_{P_{S}}:P_{S}\rightarrow \bigoplus_{S'\in \mathcal{S}_L}P_{S'}$ is the canonical embedding.\\
  Then since $\fakedelta_{L}$ has a filtration by $L'\ngtr_A L$ and $\field G\otimes -$ is exact, $\field G\otimes \fakedelta_{L}$ has a filtration by $\field G\otimes L'$ such that $L'\ngtr_A L$. Since for every $L'\ngtr_A L$ any two simple summands $S|\field G\otimes L$ and $S'|\field G\otimes L'$ fulfill $S'\ngtr_A S$, this means that no composition factor $S'$ of $\field G\otimes \fakedelta_{L}$ is greater than any summand $S$ of $\field G\otimes L$.\\
  Hence Lemma \ref{lemma_standardmodules} implies that for any $S\in \mathcal{S}_L$ there is a homomorphism $f_S:\fakedelta_{S}\rightarrow \field G\otimes \fakedelta_{L}$ such that $g_S=f_S\circ \fakepi_S$.
  Thus we obtain a homomorphism
  \begin{align*}
    f:\bigoplus_{S\in \mathcal{S}_L}\fakedelta_{S}\rightarrow \field G\otimes \fakedelta_{L}, (x_S)_S\mapsto \sum_S f_S(x_S)
  \end{align*}
  such that 
  \begin{align*}
    f\circ (\fakepi_S)_S((x_S)_S)&=\sum_S f_S\circ \fakepi_S(x_S)
    =\sum_S g_S(x_S)\\ 
    &=\sum_S(\field G\otimes \fakepi_L)\circ \varphi_L\circ \iota_{P_S}(x_S)
    =(\field G\otimes \fakepi_L)\circ \varphi_L((x_S)_S).
  \end{align*}
  In other words, the diagram 
\[\begin{tikzcd}[ampersand replacement=\&]
	{\bigoplus_S P(S)} \& {\bigoplus_S \fakedelta(S)} \\
	{\field G\otimes P(L)} \& {\field G\otimes \fakedelta(L)}
	\arrow["f", from=1-2, to=2-2]
	\arrow["{(\fakepi_S)_S}", from=1-1, to=1-2]
	\arrow["{\field G\otimes \fakepi_L}"', from=2-1, to=2-2]
	\arrow["{\varphi_L}"', from=1-1, to=2-1]
\end{tikzcd}\]
  commutes.
  On the other hand, note that  by Proposition \ref{prop_IR+RI} and Lemma \ref{lemma_gdelta=deltag} we have commutative diagram 
  \begin{center}
\[\begin{tikzcd}[ampersand replacement=\&]
	{R_G(\field G\otimes P_L)} \& {\bigoplus_{g\in G}gP_{L}} \& {\bigoplus_{g\in G}P_{gL}} \\
	{R_G(\field G\otimes \fakedelta_{L})} \& {\bigoplus_{g\in G}g\fakedelta_{L}} \& {\bigoplus_{g\in G}\fakedelta_{gL}}
	\arrow["{R_G(\field G\otimes \fakepi_L)}", from=1-1, to=2-1]
	\arrow["{\bigoplus_{g\in G}g(\fakepi_{L})}", from=1-2, to=2-2]
	\arrow["{\bigoplus_{g\in G}\fakepi_{gL}}", from=1-3, to=2-3]
	\arrow[from=2-2, to=2-3]
	\arrow[from=2-1, to=2-2]
	\arrow[from=1-1, to=1-2]
	\arrow[from=1-2, to=1-3]
\end{tikzcd}\]
  \end{center}
  where the horizontal arrows are isomorphisms. Since $R_G(\fakedelta_S)$ has a filtration by $R_G(S')$ where $S'\in \Sim(A*G)$ such that $S'\ngtr_{A*G} S$ and any simple summand $L'$ of $R_G(S')$ fulfills $S'|\field G\otimes L'$, we have for any composition factor $L'$ of $R_G(\fakedelta_S)$ that $\field G\otimes  L'\ngtr S$. Thus for $S|\field G\otimes L$, $R_G(\fakedelta_S)$ has no composition factor $L'$ such that $\field G \otimes L'$ is greater than $\field G\otimes L$ and thus no composition factor $L'$ which is greater than $L$.\\ 
  Hence, analogously to the construction of $f$, we can construct an $A$-linear map $f':R_G(\field G\otimes \fakedelta_{L})\rightarrow R_G(\bigoplus_S \fakedelta_S)$ such that
  the diagram 
\[\begin{tikzcd}[ampersand replacement=\&]
	{R_G(\field G\otimes P(L))} \&\& {R_G(\field G\otimes \fakedelta(L))} \\
	{R_G(\bigoplus_SP(S))} \&\& {R_G(\bigoplus_S\fakedelta(S))}
	\arrow["{R_G(\field G\otimes \fakepi_L)}", from=1-1, to=1-3]
	\arrow["{R_G((\fakepi_S)_S}", from=2-1, to=2-3]
	\arrow["{R_G(\varphi_L)^{-1}}", from=1-1, to=2-1]
	\arrow["{f'}"', from=1-3, to=2-3]
\end{tikzcd}\]
  commutes.\\
  We obtain a diagram
\[\begin{tikzcd}[ampersand replacement=\&]
	{R_G(\field G\otimes P(L))} \&\& {R_G(\field G\otimes \fakedelta(L))} \\
	{R_G(\bigoplus_SP(S))} \&\& {R_G(\bigoplus_S\fakedelta(S))}
	\arrow["{R_G(\field G\otimes \fakepi_L)}", from=1-1, to=1-3]
	\arrow["{R_G((\fakepi_S)_S)}", from=2-1, to=2-3]
	\arrow["{R_G(\varphi_L)^{-1}}", shift left=1, from=1-1, to=2-1]
	\arrow["{f'}", shift left=1, from=1-3, to=2-3]
	\arrow["{R_G(f)}", shift left=1, from=2-3, to=1-3]
	\arrow["{R_G(\varphi_L)}", shift left=1, from=2-1, to=1-1]
\end{tikzcd}\]
  where both squares commute, i.e. 
  \begin{align*}
    f'\circ R_G(\field G\otimes \fakepi_L)&=R_G((\fakepi_S)_S)\circ R_G(\varphi_L)^{-1}\\ 
   \textup{and } R_G(f)\circ R_G((\fakepi_S)_S)&= R_G(\field G\otimes \fakepi_L)\circ R_G(\varphi_L).
  \end{align*}
  In particular,
  \begin{align*}
    R_G(f)\circ f'\circ R_G(\field G\otimes \fakepi_L)&=R_G(f)\circ R_G((\fakepi_S)_S)\circ R_G(\varphi_L)^{-1}\\
    &= R_G(\field G\otimes \fakepi_L)\circ R_G(\varphi_L)\circ R_G(\varphi_L)^{-1}\\
    &=R_G(\field G\otimes \fakepi_L)
  \end{align*} and
  \begin{align*}
    f'\circ R_G(f)\circ R_G((\fakepi_S)_S)&=f'\circ R_G(\field G\otimes \fakepi_L)\circ R_G(\varphi_L)\\
    &=R_G((\fakepi_S)_S)\circ R_G(\varphi_L)^{-1}\circ  R_G(\varphi_L)\\
    &=R_G((\fakepi_S)_S).
  \end{align*}
  Since $R_G(\field G\otimes \fakepi_L)$ and $R_G((\fakepi_S)_S)$ are epimorphisms, this implies that $f'=R_G(f)^{-1}$. In particular, $f$ is bijective and hence an isomorphism.\\
  Thus 
  \begin{align*}
    \field G\otimes \fakedelta_L\cong \bigoplus_{S\in \mathcal{S}_L}\fakedelta_S\cong \bigoplus_{S\in \Sim(A*G)}[\field G\otimes L:S]\fakedelta_S.
  \end{align*}\qedhere
\end{proof}
\begin{corollary}\label{corollary_Deltadivides}
  $\fakedelta^{A*G}|\field G\otimes \fakedelta^A$  and  $\fakedelta^A|R_G(\fakedelta^{A*G})$.
\end{corollary}
\begin{proof}
  The first claim follows directly from the previous proposition,  and Corollary \ref{corollary_dividessimples}.
  The second claim then follows from the fact that $R_G(\fakedelta^{A*G})$ is a direct summand of $$R_G(\field G\otimes \fakedelta^A)\cong \bigoplus_{g\in G}g\fakedelta^A\cong  \bigoplus_{g\in G}\fakedelta^A,$$ so that $R_G(\fakedelta^{A*G})$ is isomorphic to a direct sum of standard modules, combined with the fact that $$\Top(R_G(\fakedelta^{A*G}))\cong R_G(\Top(\fakedelta^{A*G}))\cong R_GL^{A*G}$$ and $L^A|R_GL^{A*G}$ by Corollary \ref{corollary_dividessimples}.
\end{proof}
\begin{example}
  Let $Q$ be the quiver 
  \[\begin{tikzcd}[ampersand replacement=\&]
    1 \& 2
    \arrow["\beta", curve={height=-6pt}, from=1-2, to=1-1]
    \arrow["\alpha", curve={height=-6pt}, from=1-1, to=1-2]
  \end{tikzcd}\]
  and $A:=\field Q/J^2$, where $J$ denotes the arrow ideal. Consider the partial order $\leq$ on $\Sim(A)=\{L_1, L_2\}$ given by an antichain, i.e. all distinct elements are incomparable. Then the standard modules of $A$ are simple. Moreover, the group $G=\{e, g\}\cong \mathbb{Z}/2\mathbb{Z}$ acts on $A$ via the $\field$-linear map defined via $g(e_1)=e_2$, $g(e_2)=e_1$, $g(\alpha)=\beta$ and $g(\beta)=\alpha$,
  and the $G$ action preserves the partial order. Now by \cite[2.3]{ReitenRiedtmann}, $A*G$ is Morita equivalent to $\field [x]/(x^2)$, hence $\Delta_{\field}=\field [x]/(x^2)$ is not simple.
   In particular, neither Proposition \ref{prop_IDelta} nor Corollary \ref{corollary_Deltadivides} hold for $(\Delta_L)_L$ instead of $(\fakedelta_L)_L$.
\end{example}
Using Corollary \ref{corollary_Deltadivides} and Proposition \ref{prop_simplesofA*G}, we also obtain a concrete description of $(\fakedelta_S^{A*G})_S$. We denote, as before,  by $H_L$ the stabilizer of the isomorphism class of $L$ and choose an $H_L$-equivariant representative $L$ of this class. Moreover, we let $P_L^{A*H_L}$ be a projective cover of $L$ as an $A*H_L$-module and $\fakedelta_L^{A*H_L}$ be the corresponding pseudostandard module. Then by Corollary \ref{corollary_Deltadivides}, the restriction ${}_{A|}\fakedelta_L^{A*H_L}$ is isomorphic to a direct sum of standard modules with top $$\topp({}_{A|}\fakedelta_L^{A*H_L})\cong {}_{A|}\topp(\fakedelta_L^{A*H_L})\cong {}_{A|}L,$$ so that ${}_{A|}\fakedelta_L^{A*H_L}\cong \fakedelta_L^A$.
Thus we obtain the following corollary:
\begin{corollary}\label{corollary_deltaGmodule}
  $\fakedelta_L^A$ has the structure of an $A*H_L$-module.
\end{corollary}
\begin{lemma}\label{lemma_deltaforbasicA*G} The pseudostandard modules $\fakedelta_S^{A*G}$ for $A*G$ are of the form $\field G\otimes_{\field H_L} (\fakedelta_L\otimes V)$ where $V$ is an irreducible representation of $H_L$. Morevoer, two such modules $ \field G\otimes_{\field H_L}(\fakedelta_L\otimes V)$ and $\field G\otimes_{\field H_{L'}}( \fakedelta_{L'}\otimes W)$ of this form are isomorphic if and only if there is $g\in G$ such that $gL\cong L'$ and $gV\cong W$.
\end{lemma}
\begin{proof}
  By Proposition \ref{prop_IDelta} and Proposition \ref{prop_IR+RI}
  \begin{align*}
    \bigoplus_{S\in \Sim(A*G)}[\field G\otimes L:S]\fakedelta_{S}&\cong \field G\otimes \fakedelta_{L}
    \cong \field G\otimes_{\field H_i} \field H_i\otimes \fakedelta_{L}\\ 
    &\cong \field G\otimes_{\field H_i} (\fakedelta_{L}\otimes \field H_i)
    \cong \bigoplus_{V\in \Sim(\field H_i)}[\field H_i:V]\field G\otimes_{\field H_i}(\fakedelta_L\otimes V).
  \end{align*}
  Moreover, $\field G\otimes_{\field H_L} (\fakedelta_L\otimes V)$ has simple top $\field G\otimes_{\field H_L} (L\otimes V)$ and is in particular indecomposable.\\ 
  Thus it is isomorphic to $\fakedelta_{\field G\otimes_{\field H_L} (L\otimes V)}$.\\
  Moreover, $\field G\otimes_{\field H_L}(\fakedelta_L\otimes V)\cong \fakedelta_{\field G\otimes_{\field H_L} (L\otimes V)}$ and $\field G\otimes_{\field H_{L'}}( \fakedelta_{L'}\otimes W)\cong \fakedelta_{\field G\otimes_{\field H_{L'}} (L'\otimes W)}$ are isomorphic if and only if $\field G\otimes_{\field H_L} (L\otimes V)$ and $\field G\otimes_{\field H_{L'}} (L'\otimes W)$ are isomorphic, which, by Proposition \ref{prop_simplesofA*G}, is the case if and only if there is a $g\in G$ such that $gL\cong L'$ and $gV\cong W$.
\end{proof}
Now we use the description of the pseudostandard modules to compare properties of $(A, \leq_A)$ to $(A*G, \leq_{A*G})$.
\begin{lemma}\label{lemma_adaptedandexceptional}
  The following statements are equivalent:
  \begin{enumerate}
    \item $\leq_A$ is adapted and $(\Delta^A_L)$ is exceptional.
    \item $\leq_{A*G}$ is adapted and $(\Delta_S^{A*G})$ is exceptional.
  \end{enumerate}
\end{lemma}
\begin{proof}
     By Lemma \ref{lemma_adaptedandexceptionalrewrite} and Proposition \ref{prop_bijectionorders}, $\leq_A$ is adapted and $(\Delta_L^A)_L$ is exceptional if and only if 
      \begin{align*}
        \Hom_A(\fakedelta_L, \rad(\fakedelta_{L'}))\neq (0)\Rightarrow L<_AL'
      \end{align*}
      and 
      $\leq_{A*G}$ is adapted and $(\Delta_S^{A*G})_S$ is exceptional if and only if 
      \begin{align*}
        &\Hom_{A*G}(\fakedelta_S, \rad(\fakedelta_{S'}))\neq (0)\\ 
        \Rightarrow &L<_AL'\textup{ for some (equivalently all) }L, L' \textup{ such that }S|\field G\otimes L\textup{ and } S'|\field G\otimes L'
      \end{align*}
      Now if $\leq_A$ is adapted and $(\Delta_L^A)_L$ exceptional, then for all $S, S'\in \Sim(A*G)$ and $L, L'\in \Sim(A)$ such that $S|\field G\otimes L, S'|\field G\otimes L'$ we have
      \begin{align*}
        \Hom_{A*G}(\fakedelta_S, \rad(\fakedelta_S'))
        \subseteq  &\Hom_{A*G}(\field G\otimes \fakedelta_L, \rad(\field G\otimes \fakedelta_{L'}))\\
        \subseteq & \Hom_{A}(\field G\otimes \fakedelta_L, \rad(\field G\otimes\fakedelta_{L'}))
        \cong \bigoplus_{g, g'\in G}\Hom_A(\fakedelta_{g(L)}, \rad(\fakedelta_{g'(L')})).
      \end{align*}
      Thus if $\Hom_{A*G}(\fakedelta_S, \rad(\fakedelta_{S'}))\neq (0)$, there exist $g, g'\in G$ such that 
      \begin{align*}
        \Hom_A(\fakedelta_{g(L)}, \rad(\fakedelta_{g'(L')}))\neq (0).
      \end{align*}
       By assumption, this implies $gL<_A g'L$, and hence, since $\leq_A$ is $G$-equivariant, $L<_A L'$. Thus $\field G\otimes L<_{A*G} \field G\otimes L'$, so that $S<_{A*G}S'$.\\
      On the other hand, if $\leq_{A*G}$ is adapted and $(\Delta_S^{A*G})$ exceptional, then
      \begin{align*}
        &\Hom_A(\fakedelta_L, \rad(\fakedelta_{L'}))\neq (0)\\
        \Rightarrow &\Hom_{A*G}(\field G\otimes \fakedelta_L, \rad(\field G\otimes \fakedelta_{L'}) )\neq (0)\\
        \Rightarrow &\bigoplus_{S|\field G\otimes L, S'|\field G\otimes L'}\Hom_{A*G}(\fakedelta_S, \rad(\fakedelta_{S'})) \neq (0)\\
        \Rightarrow &\exists S|\field G\otimes L, S'|\field G\otimes L': \Hom_{A*G}(\fakedelta_S, \rad(\fakedelta_{S'}))\neq (0)\\
        \Rightarrow &\exists S|\field G\otimes L, S'|\field G\otimes L':S<_{A*G}S'\\
        \Leftrightarrow &L<_AL'.
      \end{align*}
  \end{proof}
\begin{example}\label{example_notadapted}
   It is not true that $\leq_{A*G}$ is adapted if and only if $\leq_A$ is adapted. Consider for example $A:=k[x]/(x^2)$ and $G=\{1, g\}\cong \mathbb{Z}/2\mathbb{Z}$ with $g(x)=-x$. Then $A$ has a unique simple $\field$ and the unique order is clearly $G$-equivariant and adapted. However, $A*G$ is given by the quiver 
   \begin{center}
\begin{tikzcd}
  L_1\cong \field(1/\sqrt{2}(1+g)) \arrow[rr, "\alpha=r_{-gx}", bend left] &  & L_2\cong \field(1/\sqrt{2}(1-g)) \arrow[ll, "\beta=r_{gx}", bend left]
  \end{tikzcd}
   \end{center}
   with $\alpha\beta=0=\beta\alpha$ and $\leq_{A*G}$ being an antichain. Hence $\Delta_1'=P_1$ has socle $L_2$ and top $L_1$ and $L_2$ is incomparable to $L_1$.
\end{example}
\begin{theorem}\label{thm_ghiff} 
  The following statements hold:
  \begin{enumerate}
    \item $(A, \leq_A)$ is \puhh if and only if $(A*G, \leq_{A*G})$ is \puhh.
    \item $(A, \leq_A)$ is strongly \puhh if and only if $(A*G, \leq_{A*G})$ is strongly quasi-hereditary.
    \item $(A, \leq_A)$ is directed if and only if $(A*G, \leq_{A*G})$ is directed.
  \end{enumerate}
\end{theorem}
\begin{proof}
  \begin{enumerate}
    \item Suppose $(A, \leq_A)$ is \puhh. Then $\leq_A$ is adapted and $(\Delta_L^A)$ is exceptional. So by Lemma \ref{lemma_adaptedandexceptional} $\leq_{A*G}$ is adapted and $(\Delta_S^{A*G})$ is exceptional. Moreover, $A$ has a filtration by the $\Delta_L=\fakedelta_L$. Hence $A*G\cong \field G\otimes A$ has a filtration by $\field G\otimes \fakedelta_L$. Since these decompose into a direct sum of $\fakedelta_S$, $S|\field G\otimes L$, by Proposition \ref{prop_IDelta} and Lemma \ref{lemma_adapted}, this implies that $A*G$ has a filtration by standard modules $\Delta_S=\fakedelta_S$. Hence $A*G\in F(\Delta^{A*G})$.\\
     On the other hand, suppose $A*G$ is \puhh. Then $\leq_{A*G}$ is adapted and $(\Delta_S^{A*G})$ is exceptional, so that $\leq_A$ is adapted and $(\Delta_L^A)$ is exceptional. Moreover, $A*G$ has a filtration as an $A*G$-module by standard modules. Hence as an $A$-module $A*G\cong \bigoplus_{g\in G}gA\cong |G|A$ has a filtration by the $R_G(\Delta_S)$, and by Proposition \ref{prop_IR+RI}, Lemma \ref{lemma_adapted} and Lemma  \ref{lemma_gdelta=deltag},
     \begin{align*}
      R_G(\Delta_S)=R_G(\fakedelta_S)|R_G(\field G\otimes \fakedelta_{L})\cong \bigoplus_{g\in G}\fakedelta_{g(L)}=\bigoplus_{g\in G}\Delta_{g(L)}
     \end{align*} for $S|\field G\otimes L'$,  so that, since the standard modules of $A$ are indecomposable, $R_G(\Delta_S)$ is a direct sum of standard modules and hence $|G|A\in \mathcal{F}(\Delta)$. Since by Proposition \ref{prop_addclosed}, $\mathcal{F}(\Delta)$ is closed under direct summands, $A\in \mathcal{F}(\Delta)$.
     \item By 1., $(A, \leq_A)$ is \puhh if and only if $(A*G, \leq_{A*G})$ is \puhh. Moreover, if the projective dimension of $\Delta^A$ is less than or equal to one, so is the projective dimension of $\Delta^{A*G}=\fakedelta^{A*G}|\field G\otimes \fakedelta^A=\field G\otimes \Delta^A$, and if the projective dimension of $\Delta^{A*G}$ is less than or equal to one, so is the projective dimension of $\Delta^A=\fakedelta^A|R_G(\fakedelta^{A*G})=R_G(\Delta^{A*G})$.
     \item By 1., $(A, \leq_A)$ is \puhh if and only if $(A*G, \leq_{A*G})$ is \puhh. Moreover, if $\Delta^A$ is semisimple, so is $\Delta^{A*G}=\fakedelta^{A*G}|\field G\otimes \fakedelta^A=\field G\otimes \Delta^A$, and if $\Delta_{A*G}$ is semisimple, so is $\Delta^A=\fakedelta^A|R_G(\fakedelta^{A*G})=R_G(\Delta^{A*G})$.
  \end{enumerate}
\end{proof}
\begin{remark}
  Example \ref{example_notadapted} shows that it is not in general true that if $(A, \leq_A)$ is standardly stratified so is $(A*G, \leq_{A*G})$.
\end{remark}
\begin{theorem}\label{thm_Borel}
  Let $(A, \leq_A)$ be quasi-hereditary and let $B$ be a subalgebra of $A$ such that $g(B)=B$ for all $g\in G$. Then there is a partial order $\leq_B$ on $\Sim(B)$ such that $(B, \leq_B)$ is an exact Borel subalgebra of $(A, \leq_A)$ if and only if there is a partial order $\leq_{B*G}$ on $\Sim(B*G)$ such that $(B*G, \leq_{B*G})$ is an exact Borel subalgebra of $(A*G, \leq_{A*G})$.
\end{theorem}
\begin{proof}
  First of all, note that $B*G$ is a well-defined subalgebra of $A*G$, since $g(B)\subseteq B$ for all $g\in G$, and that $A*G$ is \puhh by Theorem \ref{thm_ghiff}.\\
  Moreover, we have isomorphisms
  \begin{align*}
    A\otimes_B gX\rightarrow g(A\otimes_B X), a\otimes gx\rightarrow g\cdot(g^{-1}(a))\otimes x
  \end{align*}
  with inverses
  \begin{align*}
    g(A\otimes_B X)\rightarrow A\otimes_B gX, g(a\otimes x)\rightarrow g(a)\otimes gx
  \end{align*}
  which give rise to a natural isomorphism $A\otimes_B g-\rightarrow g\circ (A\otimes_B -)$.
  This induces natural isomorphisms $I_G\circ (A\otimes_B -)\cong ((A*G)\otimes_{B*G}-)\circ I_G$ and $(A\otimes_B -)\circ R_G\cong  R_G\circ ((A*G)\otimes_{B*G}-)$ given by the isomorphisms
  \begin{align*}
    \field G\otimes (A\otimes_B X)\rightarrow (A*G)\otimes_{B*G}(\field G\otimes X), g\otimes (a \otimes x)\mapsto (g(a)\otimes g)\otimes (e\otimes x)\\
    (A\otimes_B R_GY\rightarrow R_G((A*G)\otimes_{B*G}Y), a\otimes y\mapsto (a\otimes e)\otimes y.
  \end{align*}
  Now assume that $(B, \leq_B)$ is an exact Borel subalgebra of $A$.\\
  Dually to the induction functor $I_G(M):=\field G\otimes M$ from left $B$-modules to left $B*G$-modules defined in \ref{def_I+R}, we can consider the induction functor $I'_G(M):=M\otimes \field G$ from right $B$-modules to right $B*G$-modules, and analogously to Proposition \ref{prop_IR+RI} obtain that this preserves and reflects projectives. Now $A*G\cong A\otimes \field G=I_G'(A)$, so that since $A$ is projective as a right $B$-module, $A*G$ is projective as a right $B*G$-module.\\
  Let $(L_i^B)_{1\leq i\leq n}$ be a set of representatives of the isomorphism classes of simple $B$-modules. Then by assumption
  $L_i^A:=\Top(A\otimes_B L_i^B)$ gives rise to a set $(L_i^A)_{1\leq i\leq n}$ of representatives of the isomorphism classes of simple $A$-modules such that $A\otimes_B L_i^B\cong \Delta_{L_i^A}$ and $L_i^B\leq L_j^B$ if and only if $L_i^A\leq L_j^A$. 
In particular, since $\leq_A$ is $G$-invariant, so is $\leq_B$. Thus, it induces via Proposition \ref{prop_bijectionorders} a partial order $\leq_{B*G}$ on $\Sim(B*G)$, and by Theorem \ref{thm_ghiff}, $(B*G, \leq_{B*G})$ is directed.\\
  Set $L^A:=\bigoplus_{i}L_i^A$ and $L^B:=\bigoplus_i L_i^B$.
  Then, since by Lemma \ref{lemma_radical}
  \begin{align*}
  R_G(\field G\otimes L^A)&=R_G(\field G\otimes \Top(A\otimes_B L^B))\\
  &\cong \Top(R_G((A*G)\otimes_{B*G}(\field G\otimes L^B))),
  \end{align*} 
  we obtain an isomorphism
  \begin{align*}
    \phi:\End_B(R_G(\field G\otimes L^B))&\rightarrow \End_A(R_G(\field G\otimes L^A)),\\
     f&\mapsto \Top(\id_A\otimes f).
  \end{align*}
  This is $G$-equivariant, where the $G$-action is given by $g\cdot f(x)=gf(g^{-1}x)$.\\
  Hence this restricts to an isomorphism 
  \begin{align*}
    \End_{B*G}(\field G\otimes L^B)=\End_B(R_G(\field G\otimes L^B))^G &\rightarrow \End_{A*G}(\field G\otimes L^A)=\End_A(R_G(\field G\otimes L^A))^G,\\
    f&\mapsto \Top(\id_A\otimes f).
  \end{align*}
  This implies that any indecomposable summand $L_j^{B*G}$ of $\field G\otimes L^B$ gives rise to an indecomposable summand $L_j^{A*G}:=\Top((A*G)\otimes_{B*G}L^{B*G})$ of 
  $\Top((A*G)\otimes_{B*G}(\field G\otimes L^B)$ and that two such summands $L_j^{A*G}$ and  $L_k^{A*G}$ are isomorphic if and only if $L_j^{B*G}$ and $L_k^{B*G}$ are isomorphic.\\
  Since by Corollary \ref{corollary_dividessimples} every simple $A*G$-module is isomorphic to a summand of $\field G\otimes  L^A$ and every simple $B*G$-module is isomorphic to a summand of $\field G\otimes L^B$, this gives rise to a bijection
  \begin{align*}
    \Sim(B*G)\rightarrow \Sim(A*G)\\
    L_j^{B*G}\mapsto L_j^{A*G}=\Top((A*G)\otimes_{B*G}L_j^{B*G}).
  \end{align*}
 Moreover, by Proposition \ref{prop_IDelta}, 
  \begin{align*}
   A*G\otimes_{B*G}\field G \otimes L^B&\cong \field G\otimes A\otimes_B L^B\\
   &\cong \field G\otimes \bigoplus_{L_i^A\in \Sim(A)}\Delta_{L_i^A}\\
   &\cong \bigoplus_{L_i^A\in \Sim(A)}\bigoplus_{L_j^{A*G}\in \Sim(A*G)}[\field G\otimes L_i^A:L_j^{A*G}]\Delta_{L_j^{A*G}}.
  \end{align*}
  Since for every simple $B*G$-module $L_j^{B*G}$ we have that $A*G\otimes_{B*G}L_j^{B*G}$ is a summand of $A*G\otimes_{B*G}\field G \otimes L^B$, the module $A*G\otimes_{B*G}L_j^{B*G}$ is isomorphic to a direct sum of standard modules. Since it has indecomposable top $L_j^{A*G}$, this implies that $A*G\otimes_{B*G}L_j^{B*G}\cong \Delta_{L_j^A*G}$.\\
  Moreover,
  \begin{align*}
    L_i^{B*G}<_{B*G}L_j^{B*G}\Leftrightarrow R_G L_i^{B*G}<_{B}R_GL_j^{B*G}\Leftrightarrow \Top(A\otimes_B R_GL_i^{B*G})<_A \Top(A\otimes_B R_GL_j^{B*G})\\
    \Leftrightarrow R_G(\Top((A*G)\otimes_{B*G}L_i^{B*G}))<_A R_G(\Top((A*G)\otimes_{B*G}L_j^{B*G}))\\
    \Leftrightarrow \Top((A*G)\otimes_{B*G}L_i^{B*G})<_{A*G}\Top((A*G)\otimes_{B*G}L_j^{B*G})\Leftrightarrow L_i^{A*G}<L_j^{A*G}.
  \end{align*} 
  Thus $B*G$ is an exact Borel subalgebra of $A*G$.\\

  On the other hand, suppose that $(B*G, \leq_{B*G})$ is an exact Borel subalgebra of $(A*G, \leq_{A*G})$.\\
  Dually to the restriction functor $R_G(M)$ from left $B*G$-modules to left $B$-modules defined in \ref{def_I+R}, we can consider the restriction functor $R'_G(M)$ from right $B*G$-modules to right $B$-modules, and analogously to Proposition \ref{prop_IR+RI} obtain that this preserves and reflects projectives. Now $A$ is a direct summand of $R'_G(A*G)$ as right $B$-modules, so that since $A*G$ is projective as a right $B*G$-module, $A$ is projective as a right $B$-module.\\
  Let $(L_j^{B*G})_{1\leq j\leq m}$ be a set of representatives of the isomorphism classes of simple $B*G$-modules.
  By assumption  $L_j^{A*G}:=\Top((A*G)\otimes_{B*G}L_j^{B*G})$ gives rise to a set of representatives $(L_j^{A*G})_{1\leq j\leq m}$ of the isomorphism classes of simple $A*G$-modules such that 
  $L_{i}^{B*G}\leq_{B*G} L_j^{B*G}$ if and only if $L_{i}^{A*G}\leq_{A*G} L_j^{A*G}$ and $(A*G)\otimes_{B*G}L_j^{B*G}\cong \Delta_{L^{A*G}_j}$.\\
  In particular, since $\leq_{A*G}$ is $G$-stable, so is $\leq_{B*G}$, so that by Proposition \ref{prop_bijectionorders}, there is an induced partial order $\leq_B$ on $\Sim(B)$, and by Theorem \ref{thm_ghiff}, $(B, \leq_{B})$ is directed.\\
  Set $L^{A*G}:=\bigoplus_{j}L_j^{A*G}$ and $L^{B*G}:=\bigoplus_i L_j^{B*G}$.
  The group $G$ acts on $\End_{B}(R_GL^{B*G})$ via $g\cdot f:=gf(g^{-1}-)$. By \cite[Lemma 8]{MartinezVilla}, this induces an isomorphism
  \begin{align*}
    \theta_B:\End_{B}(R_GL^{B*G})*G&\rightarrow \End_{B*G}(L^{B*G}\otimes \field G), \\f\otimes g&\mapsto (x\otimes h\mapsto (h\cdot f)(x)\otimes hg).
  \end{align*}
  Analogously, we obtain an isomorphism 
  \begin{align*}
    \theta_A:\End_{A}(R_GL^{A*G})*G\rightarrow \End_{A*G}(L^{A*G}\otimes \field G)
  \end{align*}
  and an isomorphism
  \begin{align*}
    \theta_A':\End_{A}(R_G\Delta^{A*G})*G\rightarrow \End_{A*G}(\Delta^{A*G}\otimes \field G).
  \end{align*}
  Moreover, if
  \begin{align*}
    \alpha:R_G(\Delta^{A*G})=R_G(A*G\otimes_{B*G}L^{B*G})\rightarrow A\otimes_B R_G(L^{B*G}), (a\otimes g)\otimes x\mapsto a\otimes gx,
  \end{align*}
  denotes the canonical isomorphism, we obtain a commutative diagram
  \begin{center}
\[\begin{tikzcd}[ampersand replacement=\&]
	{\End_B(R_GL^{B*G})*G} \&\& {\End_{B*G}(L^{B*G}\otimes \field G)} \\
	{\End_A(A\otimes_B R_GL^{B*G})*G} \& {\End_{A}(R_G\Delta^{A*G})*G} \& {\End_{A*G}(\Delta^{A*G}\otimes \field G)} \\
	{\End_A(\Top(A\otimes_BR_GL^{B*G}))\otimes \field G} \& {\End_A(R_GL^{A*G})*G} \& {\End_{A*G}(L^{A*G}\otimes \field G) }
	\arrow["{\theta_A'}", from=3-2, to=3-3]
	\arrow["{\theta_A}", from=2-2, to=2-3]
	\arrow["{\theta_B}", from=1-1, to=1-3]
	\arrow["{\id_{A*G}\otimes-}", from=1-3, to=2-3]
	\arrow["{\Top^A\otimes \id_{\field G}}", from=2-2, to=3-2]
	\arrow["{\id_{A}\otimes-}", from=1-1, to=2-1]
	\arrow["{\rho_{\alpha}^{-1}\otimes \id_{\field G}}", from=2-1, to=2-2]
	\arrow["{\Top^{A*G}}", from=2-3, to=3-3]
	\arrow["{\Top^A\otimes \id_{\field G}}", from=2-1, to=3-1]
	\arrow["{\rho_{\Top^A(\alpha)}^{-1}}", from=3-1, to=3-2]
\end{tikzcd}\]
   \end{center}
  where the horizontal maps are isomorphisms and $\Top^R$ denotes the functor
  \begin{align*}
    \Top^R: \modu R&\rightarrow \modu L^R, X \mapsto \Top(X),\\
     (f:X\rightarrow Y)&\mapsto (\Top(f):\Top(X)\rightarrow \Top(Y), x+ \rad(R)X\mapsto f(x)+\rad(R)Y),
\end{align*} 
for $R\in \{A, A*G\}$.\\
Additionally, by assumption, the composition in the rightmost column is an isomorphism.
Thus so is the composition in the leftmost column, which is given by 
\begin{align*}
  \End_B(R_GL^{B*G})*G\rightarrow \End_A(\Top(A\otimes_BR_GL^{B*G}))\otimes \field G, f\otimes g\mapsto \Top(\id_A\otimes f)\otimes g
\end{align*}
This restricts to an isomorphism
\begin{align*}
  \End_B(R_GL^{B*G})\rightarrow \End_A(\Top(A\otimes_BR_GL^{B*G})) f\mapsto \Top(\id_A\otimes f).
\end{align*}
Thus we obtain a bijection between isomorphism classes of indecomposable summands of $R_GL^{B*G}$ and $\Top(A\otimes_BR_GL^{B*G})$ given by $S\mapsto \Top(A\otimes S)$. Since by Corollary \ref{corollary_dividessimples} $L^B|R_{G}L^{B*G}$ and $L^A|R_{G}L^{A*G}$, this gives rise to a bijection 
\begin{align*}
  \Sim(B)\rightarrow \Sim(A), L_i^B \mapsto L_i^A:=\Top(A\otimes_B S).
\end{align*}
  Additionally, since $A\otimes_B R_G L^{B*G}\cong R_G((A*G)\otimes_{B*G} L^{B*G})\cong R_G\Delta^{A*G}$ is isomorphic to a direct sum of standard modules by Proposition \ref{prop_IDelta} and $A\otimes_B L_i^B$ is a summand of $A\otimes_B R_G L^{B*G}$ with simple top $L_i^A$, we obtain that $A\otimes_B L_i^B\cong \Delta_i^A$.
  Moreover, we have that 
  \begin{align*}
    L_i^B<_B L_j^B\Leftrightarrow &I_G L_i^B<_{B*G} I_GL_j^B \\
    \Leftrightarrow &\Top((A*G)\otimes_{B*G}I_G L_i^B)<_{A*G} \Top((A*G)\otimes_{B*G}L_j^B)\\
    \Leftrightarrow &I_G (\Top(A\otimes_B L_i^B))<_{A*G} I_G(\Top(A\otimes_B L_j^B)) \\
    \Leftrightarrow \Top(A\otimes_B L_i^B)<_A \Top(A\otimes_B L_j^B)\Leftrightarrow &L_i^A<_A L_j^A.
  \end{align*}
  Hence $B$ is an exact Borel subalgebra of $A$.
\end{proof}
\begin{proposition}
  Let $(A, \leq_A)$ be quasi-hereditary and let $B$ be a subalgebra of $A$ such that $g(B)=B$ for all $g\in G$. Let $(B*G, \leq_{B*G})$ be the corresponding exact Borel subalgebra of $(A*G, \leq_{A*G})$
  Then the following statements hold:
  \begin{enumerate}
    \item $B$ is a strong exact Borel subalgebra if and only if $B*G$ is a strong exact Borel subalgebra.
    \item $B$ is a normal exact Borel subalgebra if and only if $B*G$ is a normal exact Borel subalgebra.
    \item $B$ is a homological exact Borel subalgebra if and only if $B*G$ is a homological exact Borel subalgebra.
    \item $B$ is a regular exact Borel subalgebra if and only if $B*G$ is a regular exact Borel subalgebra.
  \end{enumerate}
   /homological/normal/regular if and only if $B$ is strong/homological/normal/regular.
\end{proposition}
\begin{proof}
  \begin{itemize}
    \item[1.] Suppose $B$ is a strong exact Borel subalgebra.
  Then by Lemma \ref{lemma_strongBorel}, $A\rad(B)\subseteq \rad(A)$. Hence $A*G\rad(B*G)\subseteq \rad(A*G)$ by Lemma \ref{lemma_radical}, so that again by Lemma \ref{lemma_strongBorel} $B*G$ is a strong Borel subalgebra of $A*G$. 
  On the other hand, suppose that $B*G$ is a strong Borel subalgebra of $A*G$. Then $A*G\rad(B*G)\subseteq \rad(A*G)$, so that again by Lemma \ref{lemma_radical}
  \begin{align*}
    A\rad(B)\subseteq (A*G\rad(B*G))\cap A\subseteq \rad(A*G)\cap A=\rad(A).
  \end{align*}

\item[2.]   Suppose that $B$ is normal. Then the inclusion $\iota:B\rightarrow A$ has a splitting $\pi:A\rightarrow B$ as right $B$-modules whose kernel is a right ideal of $A$. Since tensoring over $\field$ is exact, the inclusion $\iota\otimes\id_{\field G}$ has the splitting $\pi\otimes \id_{\field G}$ which is a right $B*G$-module homomorphism whose kernel is a right ideal of $A*G$.\\
  On the other hand, suppose that $B*G$ is normal. Then $\iota\otimes\id_{\field G}$ has a splitting $\pi':A*G\rightarrow B*G$ of right $B*G$-modules whose kernel is a right ideal in $A*G$. Since the fixed point functor $-^G$ for the $G$-action given by left multiplication is exact by \cite[Lemma 3]{MartinezVilla}, $(\pi')^G$ is a splitting of the embedding $(\iota\otimes \id_{\field G})^G$ as right $(B*G)^G$-modules such that its kernel is a right ideal in $(A*G)^G$. Now since the upwards arrows in the commutative diagram
  \[\begin{tikzcd}[ampersand replacement=\&]
    {(B*G)^G} \&\& {(A*G)^G} \\
    B \&\& A
    \arrow["{(\iota\otimes \id_{\field G})^G}", from=1-1, to=1-3]
    \arrow["{b\mapsto \frac{1}{|G|}\sum_{g\in G}g(b)\otimes g}", from=2-1, to=1-1]
    \arrow["\iota", from=2-1, to=2-3]
    \arrow["{a\mapsto \frac{1}{|G|}\sum_{g\in G} g(a)\otimes g}"', from=2-3, to=1-3]
  \end{tikzcd}\]
  are isomorphisms of algebras, $(\pi')^G$ induces a splitting of $\iota$ as right $B$-modules such that its kernel is a right ideal in $A$.\\
\item[3.+4.] Assume that $B$ is homological, resp. regular. In the latter case, we have already seen that $B*G$ is normal.
   Let $P^B$ be a projective resolution of $L^B$. Then $A\otimes_B P^B$ is a projective resolution of $\Delta^A$, $\field G\otimes P^B$ is a projective resolution of $\field G\otimes L^B$, $\field G\otimes (A\otimes_B P^B)$ is a projective resolution of $\field G\otimes (A\otimes_B L^B)$, and similarly for the restriction of the induction. Now $G$ acts on $\End_B(R_G(\field G\otimes P^B))$ and on  $\End_A(R_G(\field G\otimes (A\otimes_B P^B)))$, and since 
    $R_G(\field G\otimes L^B)$ is semisimple, the map
    \begin{align*}
      \End_B(R_G(\field G\otimes P^B))\rightarrow \End_A(A\otimes_B R_G(\field G\otimes P^B)), f\mapsto \id_A\otimes f
    \end{align*}
    is an epimorphism in homology of degree one and an isomorphism in homology of degree strictly greater than one, resp. an isomorphism in homology of degree strictly greater than zero.\\
    Composing with the canonical isomorphism $A\otimes_B R_G(\field G\otimes P^B)\cong R_G((A*G)\otimes_{B*G}(\field G\otimes P^B))$ yields an isomorphism
    \begin{align*}
      \End_B(R_G(\field G\otimes P^B))\rightarrow \End_A(R_G(A*G)\otimes_{B*G}(\field G\otimes P^B)), f\mapsto \id_{A*G}\otimes f.
    \end{align*}
    Moreover, it is clearly $G$-equivariant, so it induces a homomorphism
    \begin{align*}
      &\End_{B*G}(\field G\otimes P^B)=\End_B(R_G(\field G\otimes P^B))^G\\ \rightarrow  &\End_{A*G}((A*G)\otimes_{B*G}(\field G\otimes P^B))=\End_A(R_G(A*G)\otimes_{B*G}(\field G\otimes P^B))^G,\\ f&\mapsto \id_{A*G}\otimes f.
    \end{align*}
    Since the fixed point functor $-^G$ is exact, this is an epimorphism in homology of degree one and an isomorphism in homology of degree strictly greater than one, resp. an isomorphism in homology of degree strictly greater than zero.\\ Thus we obtain an epimorphism in degree one and an isomorphism in degree strictly greater than one, resp. an isomorphism in degree strictly greater than zero
    \begin{align*}
      \Ext_{B*G}^{>0}(\field G\otimes L^B, \field G\otimes L^B)&\rightarrow  \Ext_{A*G}^{>0}((A*G)\otimes_{B*G}L^B, (A*G)\otimes_{B*G}L^B),\\ [f]&\mapsto [\id_{A*G}\otimes f].
    \end{align*}
    The result now follows from Corollary \ref{corollary_dividessimples}.\\
    
    On the other hand, suppose $B*G$ is homological resp. regular. In the latter case, we have already seen that $B$ is normal.
    Let $P^{B*G}$ be a projective resolution of $L^{B*G}$. Then $(A*G)\otimes_{B*G} P^{B*G}$ is a projective resolution of $\Delta_{A*G}$,  $R_GP^{B*G}$ is a projective resolution of $R_G(L^{B*G})$, and $R_G((A*G)\otimes_{B*G} P^{B*G})\cong A\otimes_B R_G(P^{B*G})$ is a projective resolution of $R_G(A*G)\otimes_{B*G} L^{B*G})\cong A\otimes_B R_G(L^{B*G})$ and similarly for the induction of the restriction.
    Since $L^{B*G}\otimes \field G$ is semisimple, we have that by assumption
    \begin{align*}
      \Ext^*_{B*G}(L^{B*G}\otimes \field G, L^{B*G}\otimes \field G)\rightarrow \Ext^*_{A*G}(\Delta^{A*G}\otimes \field G, \Delta^{A*G}\otimes \field G)
    \end{align*}
    is an epimorphism in degree one and an isomorphism in degree strictly greater than one, resp. an isomorphism in degree strictly greater than zero.\\
    Moreover, $G$ acts on $\End_{B}(R_G(P^{B*G}))$ and on $\End_A(R_G((A*G)\otimes_{B*G}P^{B*G}))$ via $g\cdot f=gf(g^{-1}-)$ and by \cite[Lemma 8]{MartinezVilla}, this induces isomorphisms of complexes
    \begin{align*}
      \theta_B: \End_B(R_G(P^{B*G}))*G&\rightarrow \End_{B*G}(P^{B*G}\otimes \field G),\\ f\otimes g&\mapsto (x\otimes h\mapsto (h\cdot f)(x)\otimes hg),
    \end{align*}
    and 
    \begin{align*}
      \theta_A: \End_A(R_G((A*G)\otimes_{B*G}P^{B*G}))*G&\rightarrow \End_{A*G}((A*G)\otimes_{B*G} P^{A*G}\otimes \field G),\\ f\otimes g&\mapsto (x\otimes h\mapsto (h\cdot f)(x)\otimes hg).
    \end{align*}
    Moreover, conjugating with the canonical isomorphism
    \begin{align*}
      \alpha: R_G(A*G\otimes_{B*G}P^{B*G})\rightarrow A\otimes_B R_G(P^{B*G}), (a\otimes g)\otimes p\mapsto a\otimes gp,
    \end{align*}
    we obtain a commutative diagram of complexes
    \begin{center}
  \[\begin{tikzcd}[ampersand replacement=\&]
    {\End_B(R_G(P^{B*G}))*G} \&\& {\End_A(A\otimes_B R_G(P^{B*G}))*G} \\
    \&\& {\End_A(R_G((A*G)\otimes_{B*G}P^{B*G}))*G} \\
    {\End_{B*G}(P^{B*G}\otimes \field G)} \&\& {\End_{A*G}((A*G)\otimes_{B*G} P^{A*G}\otimes \field G)}
    \arrow["{\theta_A}", from=2-3, to=3-3]
    \arrow["{\theta_B}"', from=1-1, to=3-1]
    \arrow["{f\mapsto \id_{A*G}\otimes f}"', from=3-1, to=3-3]
    \arrow["{\rho_{\alpha}^{-1}}", from=1-3, to=2-3]
    \arrow["{f\otimes g\mapsto \id_A\otimes f\otimes g}", from=1-1, to=1-3]
  \end{tikzcd}\]
    \end{center}
    where the vertical maps are isomorphisms and $\rho_{\alpha}$ denotes conjugation by $\alpha$.
    Taking homology, we obtain that the vector space homomorphism
    \begin{align*}
      \Ext^*_B(R_G (L^{B*G}), R_G(L^{B*G}))*G&\rightarrow \Ext^*_A(A\otimes_B R_G (L^{B*G}), A\otimes_B R_G(L^{B*G}))*G,\\
       f\otimes g&\mapsto \id_A\otimes f\otimes g,
    \end{align*}
    is an epimorphism in degree one and an isomorphism in degree strictly greater than one, resp. an isomorphism in degree strictly greater than zero.\\
    Moreover, since $A\otimes_B L_i^B=\Delta_{L_i^A}$ for every $L_i^B\in \Sim(B)$ and $$R_G(L^{B*G})\cong \bigoplus_{L_i^B\in \Sim(B)}[R_G(L^{B*G}):L_i^B]L_i^B$$
    by Proposition \ref{prop_IR+RI}, this induces for every $L_i^B, L_j^B\in \Sim(B)$ a vector space homomorphism
    \begin{align*}
      [R_G(L^{B*G}):L_i^B][R_G(L^{B*G}):L_j^B]&\Ext^*_B(L_i^B, L_j^B)\\\rightarrow &[R_G(L^{B*G}):L_i^B][R_G(L^{B*G}):L_j^B]\Ext^*_A(\Delta_{L_i^A}, \Delta_{L_j^A}),\\ f\mapsto &\id_A\otimes f,
    \end{align*}
    which is an epimorphism in degree one and an isomorphisms in degree strictly greater than one, resp. an isomorphism in degree strictly greater than zero.
    The result now follows again from Corollary \ref{corollary_dividessimples}.\qedhere
  \end{itemize}
\end{proof}
\section{Auslander algebras of Nakayama algebras}\label{s4}
In this section, we will give, at some length, an example of the above. However, before that, we will need two more general statements.\\
\begin{lemma}\label{lemma_commutativegroup}
  Let $A$ be a finite-dimensional algebra and suppose $G$ is a commutative group acting on $A$ via automorphisms. Let $[X]$ be an isomorphism class of indecomposable $A$-modules and let $H_{[X]}$ be the stabilizer of $[X]$ in $G$. Then there is a representative $Y\in [X]$ such that $Y$ has an $H_{[X]}$-action.
\end{lemma}
\begin{proof}
 Let $n:=|H_{[X]}|$.
 Let $X$ be any representative of $[X]$ and consider the $A*H_{[X]}$-module $\field H_{[X]}\otimes X$. As an $A$-module, this is isomorphic to the direct sum $nX$, so that via this isomorphism $nX$ obtains likewise the structure of an $A*H_{[X]}$-module. Thus we obtain a group homomorphism $\varphi: H_{[X]}\rightarrow \Mat_{n}(\End_A(X))^*$, where $\Mat_{n}(\End_A(X))^*$ denote the invertible elements of $\Mat_{n}(\End_A(X))^*$. Since $\End_A(X)$ is local with residue field $\field$, this induces a group homomorphism
  $\varphi':H_{[X]}\rightarrow \Gl_{n}(\field)\cong \Gl_n(\Aut(X))$. Since $G$ and thus $H_{[X]}$ is commutative, the matrices in the image have a common eigenvector. This corresponds to a summand $Y|nX$, $Y\cong X$ which is stable under the action by $H_{[X]}$ on $nX$ defined by $\varphi'$, and thus $Y$ has an $H_{[X]}$-action. 
\end{proof}
The following proposition is related to \cite[Theorem 1.3 (c) iii]{ReitenRiedtmann}.
\begin{proposition}\label{prop_auslanderalgebra}
  Let $G$ be a group acting via automorphisms on an algebra $D$ of finite representation type. Then there is an induced $G$-action on an algebra $A'$ Morita equivalent to the Auslander algebra $A$ of $D$ such that $A'*G$ is Morita equivalent to the Auslander algebra of $D*G$.\\
  Moreover, if $G$ is commutative, there is even an induced action on the Auslander algebra $A$ of $D$ such that $A*G$ is Morita equivalent to the Auslander algebra of $D*G$.
\end{proposition}
\begin{proof}
  Let $\{X\}$ be a set of representatives of the isomorphism classes of indecomposable $D$-modules, and let $M:=\bigoplus_X X$ and $N:=I_GM=\field G\otimes M$. 
  Note that every indecomposable $A*G$-module is a summand of $I_GM=N$ by Proposition \ref{prop_IR+RI}.
  By definition, $A':=\End_D(R_GN)$ is Morita equivalent to the Auslander algebra $A$ of $D$. Moreover,  as $N$ is a $D*G$-module, $G$ acts on $A'=\End_D(R_GN)^{\op}$ via conjugation. Now by \cite[Lemma 8]{MartinezVilla}, 
  \begin{align*}
    A'*G=\End_D(R_G N)^{\op}*G\cong (\End_D(R_G N)*G)^{\op} \cong \End_{D*G}(N\otimes \field G)^{\op}.
  \end{align*}
  As every indecomposable $D*G$-module is a summand of $N$, the latter is Morita equivalent to the Auslander algebra of $D*G$.\\
  Now if $G$ is commutative, choose instead a set of representatives $\{Y\}$ for the orbits of the isomorphism classes of indecomposable $D$-modules under the $G$-action which are equipped with a $H_{[Y]}$-action, according to Lemma \ref{lemma_commutativegroup}.\\
  Then the module $N':=\bigoplus_{Y\in \{Y\}}\field G\otimes_{\field H_{[Y]}} Y$ obtains the structure of a $D*G$-module and $A:=\End_D(N')$ is isomorphic to the Auslander algebra of $D$. As before, we can use \cite[Lemma 8]{MartinezVilla} to see that 
  \begin{align*}
    A*G\cong \End_D(R_GN')^{\op}*G\cong (\End_D(R_G N')*G)^{\op} \cong \End_{D*G}(I_G(R_G N'))^{\op},
  \end{align*}
  which is Morita equivalent to the Auslander algebra of $D*G$, as every indecomposable $D*G$-module is a summand of $I_GR_GN$ by Proposition \ref{prop_IR+RI}.
\end{proof}
The example we consider arises as follows. Let $D:=\field[x]/(x^N)$ and $G=\langle g| g^n\rangle\cong \mathbb{Z}/n\mathbb{Z}$ be a cyclic group with $n$ elements and generator $g\in G$.\\ 
 Consider the $G$-action on $D$ given by $gx=\xi x$ where $\xi$ is a primitive $n$-th root of unity. Then by \cite[p.241-244]{ReitenRiedtmann} $D*G$ is a self-injective Nakayama algebra with quiver $Q$
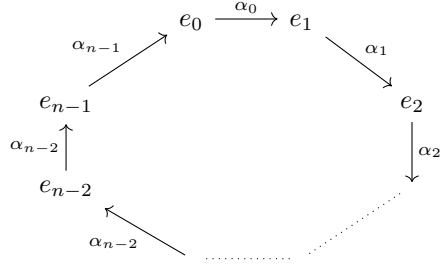
 \begin{center}
  \begin{tikzcd}
                    & e_0 \arrow[r,"\alpha_0"] & e_1 \arrow[rd,"\alpha_1"]                &                                \\
  e_{n-1} \arrow[ru,"\alpha_{n-1}"]    &               &                               & e_2 \arrow[d,"\alpha_2"]                  \\
  e_{n-2} \arrow[u,"\alpha_{n-2}"] &               &                               & {} \arrow[ld, no head, dotted] \\
                    & {} \arrow[lu,"\alpha_{n-2}"] & {} \arrow[l, no head, dotted] &                               
  \end{tikzcd}
  \end{center}
  and relations given by all paths of length $N$, i.e. $D*G=kQ/J^N$.
  The indecomposable modules of $D$ are given by $M_j=D/(x^{N-j})$ for $0\leq j\leq N-1$, and the irreducible maps between them are given by the canonical projections and embeddings
  \begin{align*}
      \pi_{j}:M_{j}\rightarrow M_{j+1}, d+(x^{N-j})\mapsto d+(x^{N-j-1})\textup{ for }0\leq j\leq N-2\\
      \iota_{j}:M_{j} \rightarrow M_{j-1}, d+(x^{N-j})\mapsto dx+(x^{N-j+1})\textup{ for }1\leq j\leq N-1\\
  \end{align*}
  with relations 
  \begin{align*}
    \pi_{j-1}\circ \iota_{j}&=\iota_{j+1}\circ \pi_{j} \textup{ for }1< j \leq N-1\\
    \pi_{N-2}\circ \iota_{N-1}&=0.
\end{align*} 
So the Auslander algebra $A$ of $D$ has quiver $Q'$
\begin{center}
\[\begin{tikzcd}[ampersand replacement=\&]
	{M_0} \& {M_1} \& \dots \& {M_{N-2}} \& {M_{N-1}}
	\arrow["{\pi_0}", curve={height=-6pt}, from=1-1, to=1-2]
	\arrow["{\iota_1}", curve={height=-6pt}, from=1-2, to=1-1]
	\arrow[curve={height=-6pt}, dotted, from=1-3, to=1-2]
	\arrow[curve={height=-6pt}, dotted, from=1-2, to=1-3]
	\arrow["{\pi_{N-2}}", curve={height=-6pt}, from=1-4, to=1-5]
	\arrow["{\iota_{N-1}}", curve={height=-6pt}, from=1-5, to=1-4]
	\arrow[curve={height=-6pt}, dotted, from=1-3, to=1-4]
	\arrow[curve={height=-6pt}, dotted, from=1-4, to=1-3]
\end{tikzcd}\]
\end{center}
with commutator relations at every middle point $M_1, \dots M_{N-2}$ and a zero relation at $M_{N-1}$.\\
Moreover, $gM_j=M_j$ for all $0\leq j\leq N_1$, so in Proposition \ref{prop_auslanderalgebra} we may chose $N':=\sum_{i=0}^{N-1}M_i$ and obtain
that $A=\End_A(N')^{\op}$ is the Auslander algebra of $D$, and $G$ acts on it via $g(\id_{M_i})=\id_{M_i}$, $g(\pi_i)=\pi_i$ and $g(\iota_i)=\xi\iota_i$.\\
Note that since $G$ acts trivially on the primitive idempotents, $A*G$ is basic, so that it is isomorphic to the Auslander algebra of $D*G$.\\
By \cite[p.241-244]{ReitenRiedtmann}, the algebra $A*G$ has Gabriel quiver $Q'$ given by
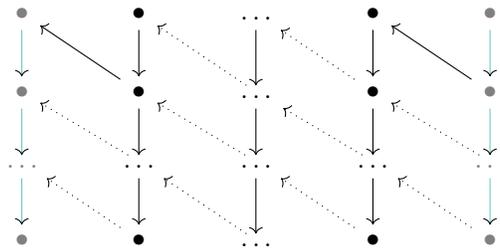
\begin{center}
\[\begin{tikzcd}[ampersand replacement=\&]
	{\color{gray}\bullet} \& \bullet \& \dots \& \bullet \& {\color{gray}\bullet} \\
	{\color{gray}\bullet} \& \bullet \& \dots \& \bullet \& {\color{gray}\bullet} \\
	\color{gray}\dots \& \dots \& \dots \& \dots \& {\color{gray}\dots} \\
	{\color{gray}\bullet} \& \bullet \& \dots \& \bullet \& {\color{gray}\bullet}
	\arrow[from=1-4, to=2-4]
	\arrow[from=2-4, to=3-4]
	\arrow[from=3-4, to=4-4]
	\arrow[from=1-3, to=2-3]
	\arrow[from=2-3, to=3-3]
	\arrow[from=2-2, to=3-2]
	\arrow[from=3-2, to=4-2]
	\arrow[from=3-3, to=4-3]
	\arrow[draw={rgb,255:gray,214;green,92;blue,92}, from=1-1, to=2-1]
	\arrow[draw={rgb,255:gray,214;green,92;blue,92}, from=3-1, to=4-1]
	\arrow[draw={rgb,255:gray,214;green,92;blue,92}, from=2-1, to=3-1]
	\arrow[dotted, from=4-2, to=3-1]
	\arrow[dotted, from=3-2, to=2-1]
	\arrow[from=2-2, to=1-1]
	\arrow[from=1-2, to=2-2]
	\arrow[dotted, from=4-3, to=3-2]
	\arrow[dotted, from=3-3, to=2-2]
	\arrow[dotted, from=2-3, to=1-2]
	\arrow[dotted, from=2-4, to=1-3]
	\arrow[dotted, from=3-4, to=2-3]
	\arrow[dotted, from=4-4, to=3-3]
	\arrow[dotted, from=4-5, to=3-4]
	\arrow[dotted, from=3-5, to=2-4]
	\arrow[from=2-5, to=1-4]
	\arrow[draw={rgb,255:gray,214;green,92;blue,92}, from=1-5, to=2-5]
	\arrow[draw={rgb,255:gray,214;green,92;blue,92}, from=2-5, to=3-5]
	\arrow[draw={rgb,255:gray,214;green,92;blue,92}, from=3-5, to=4-5]
\end{tikzcd}\]
\end{center}
where the last column is identified with the first column, so that $Q'$ becomes a cylinder, with relations given by commutator relations in every parallelogram
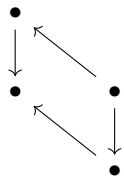
\[\begin{tikzcd}[ampersand replacement=\&]
	\bullet \\
	\bullet \& \bullet \\
	\& \bullet
	\arrow[from=2-2, to=3-2]
	\arrow[from=3-2, to=2-1]
	\arrow[from=2-2, to=1-1]
	\arrow[from=1-1, to=2-1]
\end{tikzcd}\]
and a zero relation given between neighbouring points in the last row
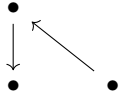
\[\begin{tikzcd}[ampersand replacement=\&]
	\bullet \\
	\bullet \& \bullet
	\arrow[from=2-2, to=1-1]
	\arrow[from=1-1, to=2-1]
\end{tikzcd}\]
Now consider $A$. The projective $P_i$ at $M_i$ is given by 
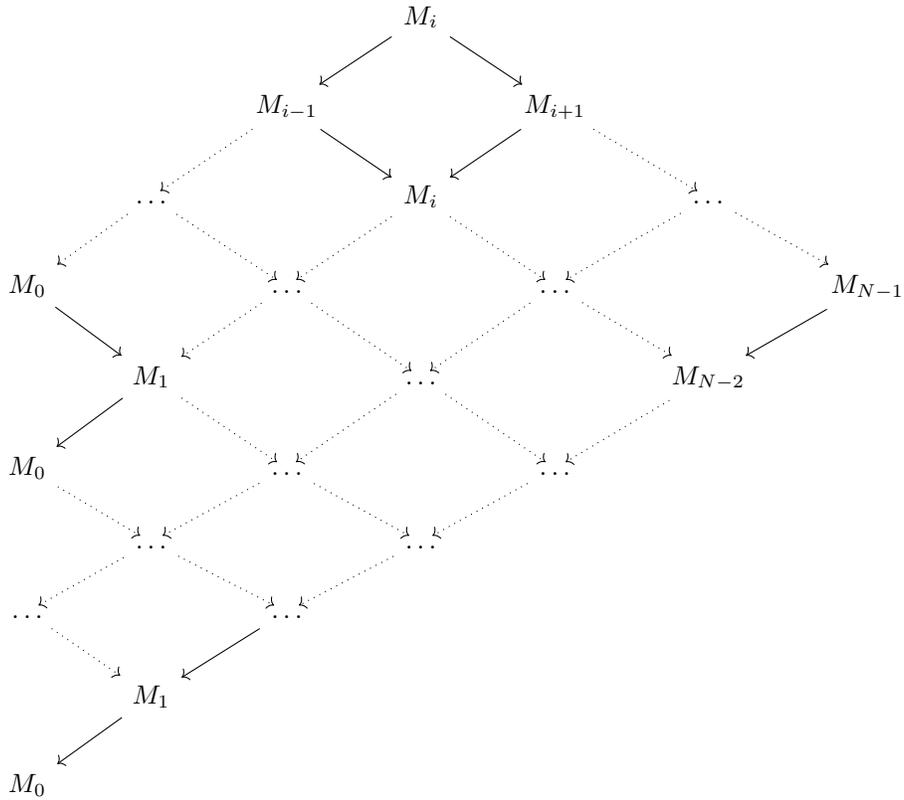
\begin{center}
\[\begin{tikzcd}[ampersand replacement=\&]
	\&\&\& {M_i} \\
	\&\& {M_{i-1}} \&\& {M_{i+1}} \\
	\& \dots \&\& {M_i} \&\& \dots \\
	{M_0} \&\& \dots \&\& \dots \&\& {M_{N-1}} \\
	\& {M_1} \&\& \dots \&\& {M_{N-2}} \\
	{M_0} \&\& \dots \&\& \dots \\
	\& \dots \&\& \dots \\
	\dots \&\& \dots \\
	\& {M_1} \\
	{M_0}
	\arrow[from=1-4, to=2-3]
	\arrow[from=2-5, to=3-4]
	\arrow[from=1-4, to=2-5]
	\arrow[dotted, from=2-5, to=3-6]
	\arrow[dotted, from=3-6, to=4-7]
	\arrow[from=4-7, to=5-6]
	\arrow[from=2-3, to=3-4]
	\arrow[dotted, from=2-3, to=3-2]
	\arrow[dotted, from=3-4, to=4-3]
	\arrow[dotted, from=3-2, to=4-3]
	\arrow[from=4-1, to=5-2]
	\arrow[from=5-2, to=6-1]
	\arrow[dotted, from=3-2, to=4-1]
	\arrow[dotted, from=4-3, to=5-2]
	\arrow[dotted, from=4-3, to=5-4]
	\arrow[dotted, from=5-6, to=6-5]
	\arrow[dotted, from=5-2, to=6-3]
	\arrow[dotted, from=5-4, to=6-5]
	\arrow[dotted, from=6-3, to=7-4]
	\arrow[dotted, from=6-1, to=7-2]
	\arrow[dotted, from=8-1, to=9-2]
	\arrow[from=9-2, to=10-1]
	\arrow[from=8-3, to=9-2]
	\arrow[dotted, from=7-2, to=8-3]
	\arrow[dotted, from=6-5, to=7-4]
	\arrow[dotted, from=7-4, to=8-3]
	\arrow[dotted, from=7-2, to=8-1]
	\arrow[dotted, from=3-4, to=4-5]
	\arrow[dotted, from=4-5, to=5-6]
	\arrow[dotted, from=6-3, to=7-2]
	\arrow[dotted, from=5-4, to=6-3]
	\arrow[dotted, from=4-5, to=5-4]
	\arrow[dotted, from=3-6, to=4-5]
\end{tikzcd}\]
\end{center}
If we define the partial order $M_i\leq M_j$ if and only if $i\leq j$, then we obtain that $A$ is quasi-hereditary with standard modules given by 
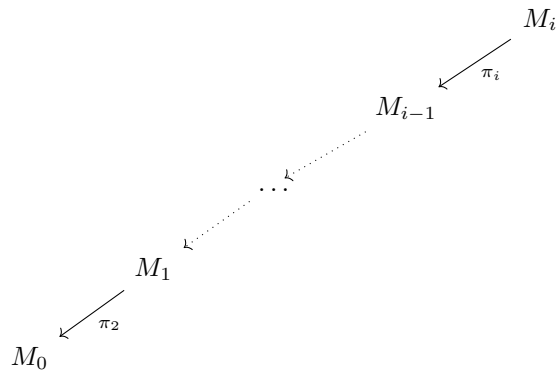
\begin{center}
\[\begin{tikzcd}[ampersand replacement=\&]
	\&\&\&\& {M_i} \\
	\&\&\& {M_{i-1}} \\
	\&\& \dots \\
	\& {M_1} \\
	{M_0}
	\arrow["{\pi_i}", from=1-5, to=2-4]
	\arrow[dotted, from=2-4, to=3-3]
	\arrow[dotted, from=3-3, to=4-2]
	\arrow["{\pi_2}", from=4-2, to=5-1]
\end{tikzcd}\]
\end{center}
where $P_i$ has standard filtration 
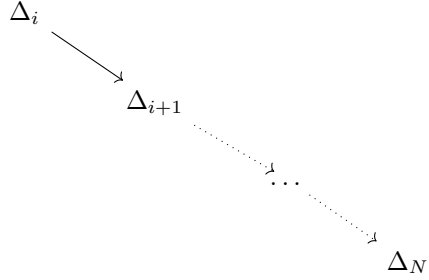
\begin{center}
\[\begin{tikzcd}[ampersand replacement=\&]
	{\Delta_i} \\
	\& {\Delta_{i+1}} \\
	\&\& \dots \\
	\&\&\& {\Delta_N}
	\arrow[from=1-1, to=2-2]
	\arrow[dotted, from=2-2, to=3-3]
	\arrow[dotted, from=3-3, to=4-4]
\end{tikzcd}\]
\end{center}
 Moreover, note that the subalgebra $B$ of $A$ given by the quiver 
\[\begin{tikzcd}[ampersand replacement=\&]
	{M_{0}} \& {M_1} \& \dots \& {M_{N-1}}
	\arrow["{\pi_0}"', from=1-1, to=1-2]
	\arrow[dotted, from=1-3, to=1-4]
	\arrow[dotted, from=1-2, to=1-3]
\end{tikzcd}\]
is directed. Additionally $A$ has a vector space basis consisting of
\begin{align*}
	\iota_{i-k}\circ\dots\circ \iota_{i}\circ \id_{M_i}\circ \pi_{i-1}\circ \dots \circ \pi_{i-j}\textup{ for }0\leq i \leq N-1; 0\leq j, k \leq i.
\end{align*}
Thus, as a right $B$-module
 $$A_{|B}=\bigoplus_{i=0}^{N-1}\bigoplus_{k=0}^{i}\iota_{i-k}\circ\dots\circ \iota_i\circ \id_{M_i} B\cong \bigoplus_{i=0}^{N-1}\bigoplus_{j=i}^{N-1}e_{i}B$$ is projective, and 
 \begin{align*}
	A\otimes_B L_{M_i}\cong A\id_{M_i}/(A\rad(B)\id_{M_i})\cong P_{M_i}/A\pi_i=\Delta_{i}.
 \end{align*}
 Thus $B$ is an exact Borel subalgebra of $A$.\\
 Now, since $G$ acts trivially on the simples of $A$, $\leq_A$ is automatically $G$-equi\-va\-riant. Hence we obtain by Proposition \ref{prop_bijectionorders} an induced partial order $\leq_{A*G}$ on $A*G$, such that by Theorem \ref{thm_ghiff} $A*G$ is \puhh. Moreover, by Theorem \ref{thm_Borel} it has an exact Borel subalgebra given by $B*G$. Using our explicit description of $B$, and the fact that $G$ acts trivially on $B$, we obtain that $B*G$ is the subalgebra of $A*G$ given by the subquiver 
 
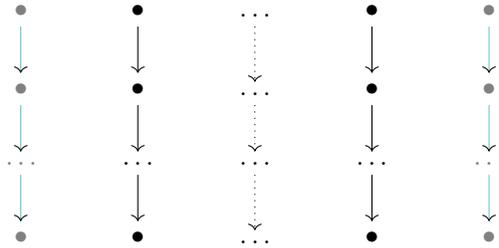
\begin{center}
\[\begin{tikzcd}[ampersand replacement=\&]
	{\color{gray}\bullet} \& \bullet \& \dots \& \bullet \& {\color{gray}\bullet} \\
	{\color{gray}\bullet} \& \bullet \& \dots \& \bullet \& {\color{gray}\bullet} \\
	\color{gray}\dots \& \dots \& \dots \& \dots \& {\color{gray}\dots} \\
	{\color{gray}\bullet} \& \bullet \& \dots \& \bullet \& {\color{gray}\bullet}
	\arrow[from=1-4, to=2-4]
	\arrow[from=2-4, to=3-4]
	\arrow[from=3-4, to=4-4]
	\arrow[dotted, from=1-3, to=2-3]
	\arrow[dotted, from=2-3, to=3-3]
	\arrow[from=2-2, to=3-2]
	\arrow[from=3-2, to=4-2]
	\arrow[dotted, from=3-3, to=4-3]
	\arrow[draw={rgb,255:gray,214;green,92;blue,92}, from=1-1, to=2-1]
	\arrow[draw={rgb,255:gray,214;green,92;blue,92}, from=3-1, to=4-1]
	\arrow[draw={rgb,255:gray,214;green,92;blue,92}, from=2-1, to=3-1]
	\arrow[from=1-2, to=2-2]
	\arrow[draw={rgb,255:gray,214;green,92;blue,92}, from=1-5, to=2-5]
	\arrow[draw={rgb,255:gray,214;green,92;blue,92}, from=2-5, to=3-5]
	\arrow[draw={rgb,255:gray,214;green,92;blue,92}, from=3-5, to=4-5]
\end{tikzcd}\]
 \end{center}
of $Q'$. Moreover, if $\iota_B: B\rightarrow A$ is the canonical embedding,  $I\subset A$ is the ideal generated by $\iota_1, \dots, \iota_{N-1}$ and $\pi_I:A\rightarrow A/I$ is the canonical projection, then $\pi_I\circ \iota_B$ is an isomorphism, so that $\iota_B$ admits a splitting with kernel $I$.
Hence $B$ is normal in $A$, so that $B*G$ is normal in $A*G$ by Theorem \ref{thm_Borel}.\\
However, note that for $N\geq 3$ the extension of $\Delta_1$ and $\Delta_2$ given by 
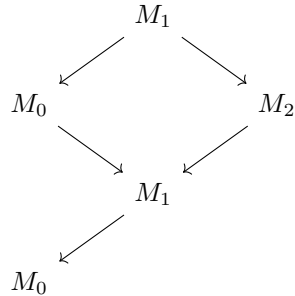
\begin{center}
\[\begin{tikzcd}[ampersand replacement=\&]
	\& {M_1} \\
	{M_0} \&\& {M_2} \\
	\& {M_1} \\
	{M_0}
	\arrow[from=1-2, to=2-3]
	\arrow[from=1-2, to=2-1]
	\arrow[from=2-3, to=3-2]
	\arrow[from=3-2, to=4-1]
	\arrow[from=2-1, to=3-2]
\end{tikzcd}\]
\end{center}
has a submodule 
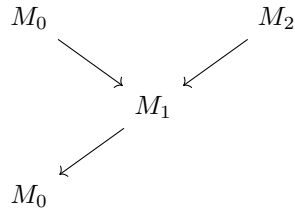
\begin{center}
\[\begin{tikzcd}[ampersand replacement=\&]
	{M_0} \&\& {M_2} \\
	\& {M_1} \\
	{M_0}
	\arrow[from=1-3, to=2-2]
	\arrow[from=2-2, to=3-1]
	\arrow[from=1-1, to=2-2]
\end{tikzcd}\]
\end{center}
which is an extension of $\Delta_0$ and $\Delta_2$. Hence $\Ext^1_A(\Delta_0, \Delta_2)\neq (0) =\Ext^1_B(S_0, S_2)$, so that $B$ is not regular. 
For the case $N=2$, $B$ is a regular exact Borel subalgebra as seen in \cite[Example A1]{KKO}.
\bibliography{cat2bibliography.bib}

\begin{thebibliography}{10}

\bibitem{Beilinson1}
Alexander Beilinson.
\newblock {C}oherent {S}heaves on {$P^n$} and {P}roblems of {L}inear {A}lgebra.
\newblock {\em Functional Analysis and its Applications}, 12:214--216, 1978.

\bibitem{Beilinson2}
Alexander Beilinson.
\newblock {T}he {D}erived {C}ategory of {C}oherent {S}heaves on {$P^n$}.
\newblock {\em Selecta Mathematica Sovietica}, 3:233--237, 1984.

\bibitem{Bondal}
A~I Bondal.
\newblock {R}epresentation of {A}ssotiative {A}lgebras and {C}oherent
  {S}heaves.
\newblock {\em Mathematics of the USSR-Izvestiya}, 34(1):23, 1990.

\bibitem{BKK}
Tomasz Brzeziński, Steffen Koenig, and Julian Külshammer.
\newblock {From Quasi-Hereditary Algebras with Exact Borel Subalgebras to
  Directed Bocses}.
\newblock {\em Bulletin of the London Mathematical Society}, 52(2):367--378,
  2020.

\bibitem{Chan}
Aaron Chan.
\newblock {Some Homological Properties of Tensor and Wreath Products of
  Quasi-Hereditary Algebras}.
\newblock {\em Communications in Algebra}, 42:2368 -- 2379, 2012.

\bibitem{ChuangKessar}
Joseph Chuang and Radha Kessar.
\newblock {Symmetric Groups, Wreath Products, Morita Equivalences, and
  Brou{\'e}'s Abelian Defect Group Conjecture}.
\newblock {\em Bulletin of the London Mathematical Society}, 34, 2002.

\bibitem{ChuangRouquier}
Joseph Chuang and Rapha{ë}l Rouquier.
\newblock Derived {E}quivalences for {Symmetric Groups} and
  {sl\_2}-{C}ategorification.
\newblock {\em Annals of Mathematics}, 167:245--298, 2008.

\bibitem{ChuangTan}
Joseph Chuang and Kai~Meng Tan.
\newblock {Filtrations in Rouquier Blocks of Symmetric Groups and Schur
  Algebras}.
\newblock {\em Proceedings of the London Mathematical Society}, 86(3):685--706,
  2003.

\bibitem{ChuangTan2}
Joseph Chuang and Kai~Meng Tan.
\newblock {Representations of Wreath Products of Algebras}.
\newblock {\em Mathematical Proceedings of the Cambridge Philosophical
  Society}, 3(135), 2003.

\bibitem{CPS}
Edward Cline, Brian Parshall, and Leonard Scott.
\newblock {Finite Dimensional Algebras and Highest Weight Categories}.
\newblock {\em Journal für die reine und angewandte Mathematik}, 391:85--99,
  1988.

\bibitem{Demonet}
Laurent Demonet.
\newblock {Skew Group Algebras of Path Algebras and Preprojective Algebras}.
\newblock {\em Journal of Algebra}, 323(4):1052--1059, 2010.

\bibitem{stratified}
Vlastimil Dlab.
\newblock {Q}uasi-hereditary {A}lgebras {R}evisited.
\newblock In {\em An. Stiin.}, volume~4, pages 43--54. Univ. Ovidius
  Constantza, 1996.

\bibitem{DlabRingel}
Vlastimil Dlab and Claus~Michael Ringel.
\newblock {T}he {M}odule {T}heoretical {A}pproach to {Q}uasi-hereditary
  {A}lgebras.
\newblock In {\em {R}epresentations of {A}lgebras and {R}elated {T}opics},
  volume 168 of {\em London Mathematical Society Lecture Note Series}.
  Cambridge University Press, 1992.

\bibitem{EvseevKleshchev}
Anton Evseev and Alexander Kleshchev.
\newblock {Blocks of Symmetric Groups, Semicuspidal KLR Algebras and Zigzag
  Schur-Weyl Duality}.
\newblock {\em Annals of Mathematics}, 118, 2016.

\bibitem{Rudakov}
A.~L. Gorodentsev and Alexei~N. Rudakov.
\newblock {E}xceptional {V}ector {B}undles on {P}rojective {S}paces.
\newblock {\em Duke Mathematical Journal}, 54:115--130, 1987.

\bibitem{Koenig}
Steffen Koenig.
\newblock {Exact {B}orel Subalgebras of Quasi-Hereditary Algebras, I}.
\newblock {\em Mathematische Zeitschrift}, 220:399--426, 1995.

\bibitem{KKO}
Steffen Koenig, Julian K{\"u}lshammer, and Sergiy Ovsienko.
\newblock {Quasi-Hereditary Algebras, Exact Borel Subalgebras,
  A-{$\infty$}-Categories and Boxes}.
\newblock {\em Advances in Mathematics}, 262:546--592, 2014.

\bibitem{MartinezVilla}
Roberto Mart{\'i}nez-Villa.
\newblock {Skew Group Algebras and their Yoneda Algebras}.
\newblock {\em Mathematical Journal of Okayama University}, 43, 2001.

\bibitem{LeMeur1}
Patrick~Le Meur.
\newblock {On the Morita Reduced Versions of Skew Group Algebras of Path
  Algebras}.
\newblock {\em The Quarterly Journal of Mathematics}, 2018.

\bibitem{LeMeur}
Patrick~Le Meur.
\newblock Crossed products of calabi-yau algebras by finite groups.
\newblock {\em Journal of Pure and Applied Algebra}, 224(10):106394, 2020.

\bibitem{Poon}
Edward Poon.
\newblock {Skew Group Rings}.
\newblock Student Project, 2016.

\bibitem{ReitenRiedtmann}
Idun Reiten and Christine Riedtmann.
\newblock {Skew Group Algebras in the Representation Theory of Artin Algebras}.
\newblock {\em Journal of Algebra}, 92(1):224--282, 1985.

\bibitem{Ringel}
Claus~Michael Ringel.
\newblock {Iyama's Finiteness Theorem via Strongly Quasi-Hereditary Algebras}.
\newblock {\em Journal of Pure and Applied Algebra}, 214:1687--1692, 2009.

\bibitem{Scott}
Leonard Scott.
\newblock {Simulating Algebraic Geometry with Algebra, I: The Algebraic Theory
  of Derived Categories}.
\newblock In {\em The {A}rcata {C}onference on {R}epresentations of {F}inite
  {G}roups}, volume~47 of {\em Proc. Symp. Pure Math}, pages 271--281. Amer.
  Math. Soc., 1987.

\bibitem{witherspoon}
Sarah Witherspoon.
\newblock {\em {H}ochschild {C}ohomology for {A}lgebras}.
\newblock Graduate studies in Mathematics 204. American Mathematical Society,
  2019.

\bibitem{skewcalabi}
Quanshui Wu and Can Zhu.
\newblock {Skew Group Algebras of Calabi–Yau Algebras}.
\newblock {\em Journal of Algebra}, 340(1):53--76, 2011.

\end{thebibliography}
\end{document}